\documentclass[a4paper,11pt,reqno]{amsart}

\usepackage{amssymb}
\usepackage[all]{xy}

\usepackage{hyperref}
\usepackage{xcolor}
\usepackage{enumitem}
\usepackage{subcaption}
\usepackage[final]{graphicx}
\usepackage{float}
\usepackage{epic}
\usepackage{setspace}

\usepackage{calligra}   
\usepackage[all]{xy}
\usepackage{color}

\usepackage{verbatim}

\usepackage{tikz}
\usetikzlibrary[topaths]

\newcount\mycount

\numberwithin{equation}{section}

\numberwithin{equation}{section}
\newtheorem{introtheorem}{Theorem}

\newtheorem{introcorollary}[introtheorem]{Corollary}

\newtheorem{theorem}{Theorem}[section]
\newtheorem{lemma}[theorem]{Lemma}
\newtheorem{proposition}[theorem]{Proposition}

\newtheorem*{theorem*}{Theorem}
\newtheorem*{corollary*}{Corollary}

\theoremstyle{definition}
\newtheorem{definition}[theorem]{Definition}

\theoremstyle{remark}

\theoremstyle{remark}

\newtheorem{remark}[theorem]{Remark}
\newtheorem{question}[theorem]{Question}

\begin{document}

\title{Quantum metric structures on Iwahori--Hecke algebras}

\author{Mario Klisse}

\author{Helena Perovi\'{c}}

\address{KU Leuven, Department of Mathematics,
Celestijnenlaan 200B, 3001 Leuven, Belgium}

\email{mario.klisse@kuleuven.be}

\email{helena.perovic@student.kuleuven.be}

\date{\today. \emph{MSC2010:}  20C08, 46L87, 20F55, 20F65. The first author is supported by the FWO postdoctoral grant 1203924N of the Research Foundation Flanders.}

\maketitle

\begin{abstract}
Iwahori–Hecke algebras are $q$-deformations of group algebras of Coxeter groups. In this article, we initiate a systematic study of quantum metric structures on Iwahori–Hecke algebras by establishing that, for finite rank right-angled Coxeter systems, the canonical filtrations of the corresponding Iwahori–Hecke algebras satisfy the Haagerup-type condition introduced by Ozawa and Rieffel if and only if the Coxeter diagram’s complement contains no induced squares. As a consequence, these algebras naturally inherit compact quantum metric space structures in the sense of Rieffel. Additionally, we investigate continuity phenomena in this framework by demonstrating that, as the deformation parameter $q$ approaches 1, the deformed Iwahori–Hecke algebras converge to the group algebra of the Coxeter group in Latrémolière’s quantum Gromov–Hausdorff propinquity.
\end{abstract}

\maketitle

\section{Introduction}

\vspace{3mm}

Iwahori–Hecke algebras are deformations of group algebras associated with Coxeter groups, playing a pivotal role in a variety of mathematical disciplines, including representation theory \cite{Bourbaki02}, algebraic geometry \cite{KazhdanLusztig79, Kumar02}, knot theory \cite{Jones87}, combinatorics \cite{HivertThiery09, HivertSchillingThiery13}, and harmonic analysis \cite{Opdam04, DelormeOpdam08, OpdamSolleveld10, Solleveld12, Solleveld18}. First introduced by Iwahori and Matsumoto \cite{Iwahori64, IwahoriMatsumoto65}, these algebras emerged in the context of the representation theory of finite groups of Lie type and $p$-adic reductive groups. Their structure is particularly well understood for spherical and affine Coxeter systems (see \cite{IwahoriMatsumoto65, KazhdanLusztig79, Bernstein84, KazhdanLusztig87}), while in the setting of infinite non-affine Coxeter groups, they naturally arise in connection with buildings and Kac–Moody groups acting on them \cite{Remy12}.

Beyond their algebraic features, Iwahori–Hecke algebras naturally give rise to rich analytic structures. Given any Coxeter system $(W, S)$ and deformation parameter $q$, the corresponding Iwahori--Hecke algebra $\mathbb{C}_{q}[W]$ admits a natural representation on the Hilbert space of square-summable complex functions on $W$, leading to C$^\ast$-algebraic and von Neumann-algebraic completions $C_{r,q}^{\ast}(W)$ and $\mathcal{N}_{q}(W)$, respectively. Such completions were considered early on by Matsumoto \cite{Matsumoto77} in the case of spherical and (extended) affine Coxeter groups. Later, in the 1980s, Baum, Higson, and Plymen investigated affine-type Iwahori--Hecke algebras within the framework of the Baum–Connes conjecture (see \cite{BaumConnesHigson94}). The work of Opdam \cite{Opdam04} further deepened the analytic study of these algebras, inspiring a series of subsequent contributions \cite{DelormeOpdam08, OpdamSolleveld10, Solleveld12, Solleveld18}.

In the infinite non-affine setting, Hecke operator algebras first appeared in studies on the $\ell^{2}$-cohomology of buildings \cite{Dymara06, DavisDymaraJanuszkiewiczOkun07, Davis08}. Motivated by these developments, Davis posed in \cite[Chapter 19]{Davis08} the problem of classifying factorial Hecke-von Neumann algebras. A classification in the right-angled single-parameter case was achieved by Garncarek \cite{Garncarek16}, who also established a connection between these algebras and Dykema's interpolated free group factors \cite{Dykema93, Radulescu94}; Raum and Skalski subsequently extended these results to the multi-parameter case \cite{RaumSkalski23}. Since then, Hecke operator algebras have been the subject of extensive investigation both at the C$^{\ast}$-algebraic and von Neumann-algebraic levels \cite{CaspersSkalskiWasilewski19, Caspers20, CaspersKlisseLarsen21, RaumSkalski22, Klisse23-1, Klisse23-2, RaumSkalski23, BorstCaspersWasilewski24, Klisse25}.

The aim of this article is to initiate the study of quantum metric structures on Iwahori--Hecke algebras from the viewpoint of compact quantum metric spaces. A \emph{compact quantum metric space} $(A, L)$ consists of an order unit space $A$ with a distinguished unit $e$, equipped with a densely defined seminorm $L$ -- called the \emph{Lip-norm} -- such that $\ker(L) = \mathbb{R}e$, and such that the metric
\[
\rho_{L}(\phi, \psi) := \sup \{ |\phi(x) - \psi(x)| \mid L(x) \leq 1 \}
\]
induces the weak$^{\ast}$-topology on the state space of $A$. This concept was introduced by Rieffel in \cite{Rieffel98, Rieffel99, Rieffel04}, building on ideas from Connes’ noncommutative geometry \cite{Connes89, Connes94} and motivations arising from high-energy physics.

Our primary motivation is two-fold. On one hand, we aim to construct and study novel classes of compact quantum metric spaces within the context of Iwahori--Hecke algebras. On the other hand, we emphasize that the perspective of noncommutative metric geometry opens promising new directions in the study of Iwahori--Hecke algebras, meriting further systematic exploration.

Following a framework introduced by Ozawa and Rieffel \cite{OzawaRieffel05}, given a finite rank Coxeter system $(W, S)$ and a multi-parameter $q$, we define a natural $\ast$-filtration of $\mathbb{C}_{q}[W]$ by finite-dimensional subspaces. This filtration induces a spectral triple $(\mathbb{C}_{q}[W], \ell^{2}(W), D_{S})$, where $D_{S} := \sum_{n \in \mathbb{N}} n P_{n}$ and $P_{n}$ denotes the orthogonal projection onto the span of reduced words of length $n$ in $\ell^{2}(W)$. Analogous constructions for group C$^{\ast}$-algebras were explored in \cite{Rieffel02, OzawaRieffel05, ChristRieffel}.

By blending techniques from \cite{OzawaRieffel05} with generator decompositions in graph products developed in \cite{CaspersKlisseLarsen21}, we establish a combinatorial condition characterizing when the canonical filtrations of right-angled Iwahori--Hecke algebras satisfy a \emph{Haagerup-type condition} in the sense of Ozawa–Rieffel.

\begin{introtheorem}[{Theorem \ref{HaagerupTypeCondition}}] \label{TheoremA}
Let $(W, S)$ be a finite rank, right-angled Coxeter system, and let $\Gamma$ denote the finite undirected simplicial graph with vertex set $S$ and edge set $\{ (s, t) \in S \times S \mid m_{s,t} = 2 \}$. Then the following statements hold:
\begin{enumerate}
    \item If $\Gamma$ contains no induced square, there exists a constant $K > 0$ such that
    \[
    \| P_{i} x P_{j} \| \leq K C_{q} \| x \delta_{e} \|_{2}
    \]
    for all multi-parameters $q$ and $i, j, n \in \mathbb{N}$, $x \in \chi_{n}(\mathbb{C}_{q}[W])$, where $\chi_{n}$ denotes the word length projection onto reduced words of length $n$ and
    \[
    C_{q} := \max_{\Gamma_{0} \in \mathrm{Cliq}(\Gamma)} \left( \prod_{t \in V\Gamma_{0}} |p_{t}(q)| \right).
    \] \label{1}
    \item If $\Gamma$ contains an induced square, then for any multi-parameter $q$, no constant $K > 0$ exists such that
    \[
    \| P_{i} x P_{j} \| \leq K \| x \delta_{e} \|_{2}
    \]
    holds uniformly for all $i, j, n \in \mathbb{N}$ and $x \in \chi_{n}(\mathbb{C}_{q}[W])$.
\end{enumerate}
\end{introtheorem}

As an immediate consequence of \cite[Sections 2 and 3]{OzawaRieffel05}, Iwahori--Hecke algebras associated with Coxeter groups satisfying condition \eqref{1} above can be endowed with compact quantum metric structures.

\begin{introcorollary}[{Theorem \ref{CQMSStatement}}] \label{CQMSStatement-1}
Let $(W, S)$ be a finite rank, right-angled Coxeter system, and let $q$ be a multi-parameter such that the graph $\Gamma$ defined in Theorem \ref{TheoremA} contains no induced square. Then, the pair $(C_{r, q}^{\ast}(W), L_{S}^{(q)})$ with
\[
L_{S}^{(q)}(x) :=
\begin{cases}
\| [D_{S}, x] \|, & \text{if } x \in \mathbb{C}_{q}[W], \\
\infty, & \text{otherwise},
\end{cases}
\]
defines a compact quantum metric space.
\end{introcorollary}

One of the striking features of equipping C$^{\ast}$-algebras with noncommutative metric structures is the potential to formulate quantum analogues of the classical Gromov–Hausdorff distance. Such notions enable the study of convergence phenomena for noncommutative spaces, extending techniques from classical metric geometry to the quantum realm.

The first definition of a quantum Gromov–Hausdorff distance was proposed by Rieffel \cite{Rieffel04}, leading to a surge of activity and various alternative constructions \cite{Kerr03, Li03, Rieffel04, Li06, KerrHanfeng09, Latremoliere15, Latremoliere16, Latremoliere19, Latremoliere22}. In this paper, we adopt Latrémolière's \emph{quantum Gromov–Hausdorff propinquity} \cite{Latremoliere16}, which refines Rieffel's metric by incorporating Leibniz seminorms, thereby ensuring compatibility with algebraic structures.

By developing new methods based on positive definite functions on Coxeter groups, we establish the following result.

\begin{introtheorem}[{Theorem \ref{ConvergenceTheorem}}] \label{IntroductionTheorem2}
Let $(W, S)$ be a finite rank, right-angled Coxeter system such that the graph $\Gamma$ defined in Theorem \ref{HaagerupTypeCondition} contains no induced square. Then,
\[
(C_{r, q}^{\ast}(W), L_{S}^{(q)}) \to (C_{r}^{\ast}(W), L_{S}^{(1)})
\]
in the quantum Gromov–Hausdorff propinquity as $q \to 1$.
\end{introtheorem}

While the continuity of $q$-deformed spaces in the deformation parameter has been explored in other contexts \cite{Rieffel04, AguilarKaadKyed22}, explicit examples remain scarce. Moreover, as highlighted in \cite{AguilarHartglassPenneys22}, most existing results on quantum Gromov–Hausdorff convergence rely on finite-dimensional approximations. Our approach, in contrast, makes no use of such approximations, which we believe adds conceptual significance to our methods beyond the present setting.

\vspace{3mm}

\noindent \emph{Structure}. In Section \ref{Section1}, we recall essential preliminaries on graph theory, Coxeter groups, and compact quantum metric spaces. Section \ref{Section2} introduces natural $\ast$-filtrations on Iwahori--Hecke algebras and provides a characterization of the Haagerup-type condition for right-angled Coxeter groups. This leads to a strengthened version of Corollary \ref{CQMSStatement-1}, which we utilize in Section \ref{Section3}. There, we discuss Latrémolière’s quantum Gromov–Hausdorff propinquity and develop techniques involving Schur multipliers associated with positive definite functions on Coxeter groups, culminating in the proof of Theorem \ref{IntroductionTheorem2}. Section \ref{Section4} concludes with a discussion of our main results and outlines potential avenues for future research.

\vspace{3mm}


\section{Preliminaries} \label{Section1}

\vspace{3mm}

\subsection{General Notation}

We will write $\mathbb{N}:=\left\{ 0,1,2,...\right\} $ and $\mathbb{N}_{\geq 1}:=\left\{ 1,2,...\right\} $ for the natural numbers. The neutral element of a group is always denoted by $e$ and for a set $S$ we write $\#S$ for the number of elements in $S$.

The state space of a C$^{\ast}$-algebra $A$ is denoted by $\mathcal{S}(A)$, while $\mathfrak{sa}(A)$ denotes the self-adjoint elements of $A$. The bounded operators on a Hilbert space $\mathcal{H}$ are denoted by $\mathcal{B}(\mathcal{H})$. Furthermore, we write $p^{\perp}:=1-p$ for a projection $p$.

\vspace{3mm}


\subsection{Graphs\label{subsec:Graphs}}

Given a graph $\Gamma$, we denote its \emph{vertex set} by $V\Gamma$ and its \emph{edge set} by $E\Gamma$. Unless stated otherwise, all graphs considered in this paper are assumed to be finite, undirected, and \emph{simplicial} -- that is, $E\Gamma \subseteq (V\Gamma \times V\Gamma) \setminus \{(v,v) \mid v \in V\Gamma\}$.

For a vertex $v \in V\Gamma$, the \emph{link} of $v$, denoted $\mathrm{Link}(v)$, is the set of vertices $v' \in V\Gamma$ such that $(v, v') \in E\Gamma$. Furthermore, for a subset $X \subseteq V\Gamma$, we define its \emph{common link} as $\mathrm{Link}(X) := \bigcap_{v \in X} \mathrm{Link}(v)$, with the convention that $\mathrm{Link}(\emptyset) := V\Gamma$.

A \emph{clique} in $\Gamma$ is a subgraph $\Gamma_0 \subseteq \Gamma$ in which every pair of distinct vertices is connected by an edge. We write $\mathrm{Cliq}(\Gamma)$ for the set of all cliques in $\Gamma$, and $\mathrm{Cliq}(\Gamma, l)$ for the collection of cliques consisting of exactly $l$ vertices.

Given a subgraph $\Gamma_{0}\subseteq\Gamma$, we denote by $\mathrm{Comm}(\Gamma_{0})$ the set of disjoint pairs $(\Gamma_{1},\Gamma_{2})\in\mathrm{Cliq}(\Gamma)\times\mathrm{Cliq}(\Gamma)$ contained in the link of $\Gamma_{0}$.

\vspace{3mm}


\subsection{Coxeter Groups}

A \emph{Coxeter group} is a group $W$ that admits a presentation of the form
\[
W = \left\langle S \,\middle|\, (st)^{m_{s,t}} = e \text{ for all } s, t \in S \right\rangle,
\]
where $S$ is a (possibly infinite) generating set, and the exponents $m_{s,t} \in \{1, 2, \ldots, \infty\}$ satisfy $m_{s,s} = 1$ and $m_{s,t} \geq 2$ for all distinct $s, t \in S$. A relation of the form $(st)^m = e$ is included only when $m_{s,t} < \infty$; the case $m_{s,t} = \infty$ indicates that no relation is imposed for the corresponding pair. The pair $(W, S)$ is referred to as a \emph{Coxeter system}. The system has \emph{finite rank} if $S$ is finite, and it is said to be \emph{right-angled} if $m_{s,t} \in \{2, \infty\}$ for all $s \ne t$, meaning that any two distinct generators either commute or generate an infinite dihedral group.

In the right-angled setting, a cancellation of the form
\[
s_1 \cdots s_n = s_1 \cdots \widehat{s_i} \cdots \widehat{s_j} \cdots s_n \quad \text{for } s_1, \ldots, s_n \in S
\]
implies that $s_i = s_j$ and that $s_i$ commutes with the intervening generators $s_{i+1}, \ldots, s_{j-1}$. For a detailed exposition on Coxeter groups, we refer the reader to~\cite{Davis08}.

Given a Coxeter system $(W, S)$, we denote the associated word length function by $|\cdot|$. An expression $\mathbf{w} = s_1 \cdots s_n$ with $s_i \in S$ is called \emph{reduced} if $n = |\mathbf{w}|$. For elements $\mathbf{v}, \mathbf{w} \in W$, we say that $\mathbf{w}$ \emph{starts in} $\mathbf{v}$ if $|\mathbf{v}^{-1}\mathbf{w}| = |\mathbf{w}| - |\mathbf{v}|$. In this case, we write $\mathbf{v} \leq_R \mathbf{w}$. This relation defines a partial order on $W$, known as the \emph{right weak Bruhat order}, under which $W$ becomes a \emph{complete meet-semilattice} (see~\cite[Proposition~3.2.1]{BjoernerBrenti05}). For brevity, we will usually write $\leq$ instead of $\leq_R$.

For any subset $T \subseteq S$, the subgroup $W_T := \langle T \rangle$ generated by $T$ is called a \emph{special subgroup}. According to~\cite[Theorem~4.1.6]{Davis08}, $W_T$ is again a Coxeter group with the same exponents as $W$; explicitly,
\[
W_T \cong \left\langle T \,\middle|\, (st)^{m_{s,t}} = e \text{ for all } s, t \in T \right\rangle,
\]
with the isomorphism being canonical.

Moussong provided a fundamental characterization of word-hyperbolic Coxeter groups in the sense of Gromov (see~\cite{Gromov87}):

\begin{theorem}[{\cite[Theorem 17.1]{Moussong88}}] \label{MoussongTheorem}
Let $(W, S)$ be a Coxeter system of finite rank. Then the following statements are equivalent:
\begin{enumerate}
    \item $W$ is \emph{word-hyperbolic}; that is, there exists a constant $\delta \geq 0$ such that
    \[
    |\mathbf{w}_1^{-1} \mathbf{w}_2| + |\mathbf{w}_3^{-1} \mathbf{w}_4| \leq \max\left\{ |\mathbf{w}_1^{-1} \mathbf{w}_3| + |\mathbf{w}_2^{-1} \mathbf{w}_4|,\; |\mathbf{w}_1^{-1} \mathbf{w}_4| + |\mathbf{w}_2^{-1} \mathbf{w}_3| \right\} + \delta
    \]
    for all $\mathbf{w}_1, \mathbf{w}_2, \mathbf{w}_3, \mathbf{w}_4 \in W$.
    
    \item $W$ contains no subgroup isomorphic to $\mathbb{Z}^2$.
    
    \item There is no subset $T \subseteq S$ with $\#T \geq 3$ such that $W_T$ is infinite and virtually Abelian, and there do not exist disjoint subsets $T_1, T_2 \subseteq S$ such that $W_{T_1}$ and $W_{T_2}$ are infinite and commute.
\end{enumerate}
\end{theorem}

In the case of right-angled Coxeter systems, Moussong’s theorem leads to an especially tractable characterization of word-hyperbolicity in terms of the complement of the (unlabeled) Coxeter diagram. Let $(W, S)$ be a right-angled Coxeter system, and let $\Gamma$ be the finite, undirected, simplicial graph with vertex set $V\Gamma := S$ and edge set $E\Gamma := \{ (s, t) \in S \times S \mid m_{s,t} = 2 \}$. Then $W$ is word-hyperbolic if and only if $\Gamma$ contains no \emph{induced square}; that is, there do not exist generators $s_1, s_2, s_3, s_4 \in S$ such that
\[
m_{s_1, s_2} = m_{s_2, s_3} = m_{s_3, s_4} = m_{s_4, s_1} = 2 \quad \text{and} \quad m_{s_1, s_3} = m_{s_2, s_4} = \infty.
\]

\vspace{3mm}


\subsection{Compact Quantum Metric Spaces}

The concept of compact quantum metric spaces was introduced by Rieffel in \cite{Rieffel98,Rieffel99,Rieffel04} within the general framework of order unit spaces. A \emph{compact quantum metric space} $(A,L)$ consists of an order unit space $A$ together with a densely defined seminorm $L$ on $A$, such that the dual seminorm induces a metric on the state space $\mathcal{S}(A)$ which metrizes the weak$^{\ast}$-topology. While this level of generality affords flexibility and conceptual elegance, our focus in the present work will be restricted to the C$^{\ast}$-algebraic setting, which suffices for our purposes.

\begin{definition}[{\cite[Definition 2.2]{Rieffel04}}] \label{CQMSDefinition}
Let $A$ be a unital C$^{\ast}$-algebra and let $L\colon \mathfrak{sa}(A)\rightarrow [0,\infty]$ be a seminorm. The pair $(A,L)$ is called a \emph{compact quantum metric space} if the following conditions are satisfied:
\begin{enumerate}
    \item $L(a)=0$ if and only if $a\in\mathbb{R}1$.
    \item The domain $\text{dom}(L) := \{a \in \mathfrak{sa}(A) \mid L(a) < \infty\}$ is dense in $\mathfrak{sa}(A)$.
    \item The \emph{Monge--Kantorovich metric} $d_L$ on $\mathcal{S}(A)$, defined by
    \[
    d_L(\varphi,\psi) := \sup\{|\varphi(a)-\psi(a)| \mid a \in \mathfrak{sa}(A),\, L(a) \leq 1\},
    \]
    metrizes the weak$^{\ast}$-topology on $\mathcal{S}(A)$.
\end{enumerate}
In this context, $L$ is referred to as a \emph{Lip-norm}.
\end{definition}

The Monge--Kantorovich metric had previously appeared in the work of Connes \cite{Connes89,Connes94} in connection with spectral triples, where it served as a tool to study metric properties in non-commutative geometry.
 
Rieffel, and later Ozawa and Rieffel provided a useful characterization of compact quantum metric spaces in terms of total boundedness \cite{Rieffel98,OzawaRieffel05}.

\begin{definition}[{\cite[Theorem 1.8]{Rieffel98} and \cite[Proposition 1.3]{OzawaRieffel05}}] \label{CQMSCharacterization}
Let $A$ be a unital C$^{\ast}$-algebra and $L\colon \mathfrak{sa}(A)\rightarrow[0,\infty]$ a seminorm with dense domain and $\ker(L) = \mathbb{R}1$. Then the following conditions are equivalent:
\begin{enumerate}
    \item The pair $(A,L)$ is a compact quantum metric space.
    \item The image of the set
    \[
    \mathcal{L}_1(A,L) := \{a \in \text{dom}(L) \mid L(a) \leq 1\}
    \]
    is totally bounded in the quotient space $A / \mathbb{C}1$.
    \item There exists a state $\varphi \in \mathcal{S}(A)$ such that the set
    \[
    \mathcal{L}_1(A,L,\varphi) := \{a \in \text{dom}(L) \mid \varphi(a) = 0,\, L(a) \leq 1\}
    \]
    is totally bounded in $A$.
\end{enumerate}
\end{definition}

In addition to generalizing the notion of classical compact metric spaces (see \cite{Rieffel98,Rieffel99}), Definition~\ref{CQMSDefinition} plays a central role in the construction of a non-commutative version of the Gromov--Hausdorff distance, developed in \cite{Rieffel04}. A refined variant of this distance, known as the \emph{quantum Gromov--Hausdorff propinquity}, was introduced by Latrémolière (see Section~\ref{qGHContinuity}). It is defined on a distinguished subclass of compact quantum metric spaces, the so-called \emph{Leibniz quantum compact metric spaces}, which satisfy additional compatibility conditions between the Lip-norm and the C$^{\ast}$-algebra structure.

\begin{definition}[{\cite[Definition 2.19]{Latremoliere16}}]
Let $(A,L)$ be a compact quantum metric space. We say that $(A,L)$ is a \emph{Leibniz quantum compact metric space} if the following conditions hold:
\begin{enumerate}
    \item The seminorm $L$ satisfies the \emph{Leibniz property}, i.e., for all $a,b \in \mathfrak{sa}(A)$,
    \[
    L(a \circ b) \leq \|a\| L(b) + \|b\| L(a) \quad \text{and} \quad
    L(\{a, b\}) \leq \|a\| L(b) + \|b\| L(a),
    \]
    where $a \circ b := \frac{1}{2}(ab + ba)$ and $\{a,b\} := \frac{1}{2i}(ab - ba)$.
    \item $L$ is \emph{lower semi-continuous} with respect to the operator norm, i.e., for every $r > 0$, the set $\{a \in \mathfrak{sa}(A) \mid L(a) \leq r\}$ is norm-closed.
\end{enumerate}
In this case, $L$ is referred to as a \emph{Leibniz Lip-norm}.
\end{definition}

\vspace{3mm}


\section{Iwahori--Hecke Algebras as Compact Quantum Metric Spaces\label{sec:Iwahori--Hecke-Algebras}} \label{Section2}

\vspace{3mm}

Iwahori--Hecke algebras, introduced by Iwahori and Matsumoto in \cite{Iwahori64,IwahoriMatsumoto65}, arise as deformations of group algebras of Coxeter groups and play a pivotal role in representation theory, algebraic geometry, combinatorics, and number theory. Originating in the 1950s as double coset algebras in the study of $p$-adic groups, these algebras encapsulate deep connections between algebraic, geometric, and combinatorial structures.

\begin{definition}[{\cite[Subsection 7.1]{Humphreys90} and \cite[Proposition 19.1.1]{Davis08}}] \label{IwahoriHeckeDefinition}
Let $(W,S)$ be a Coxeter system and $R$ a commutative unital ring. For any $q=(q_s)_{s \in S} \in R^S$ with $q_s = q_t$ whenever $s$ and $t$ are conjugate in $W$, the associated \emph{Iwahori--Hecke algebra} $R_q[W]$ is the free $R$-module with basis $\{\widetilde{T}^{(q)}_{\mathbf{w}} \mid \mathbf{w} \in W\}$, with multiplication defined by
\[
\widetilde{T}^{(q)}_s \widetilde{T}^{(q)}_{\mathbf{w}} = 
\begin{cases}
\widetilde{T}^{(q)}_{s\mathbf{w}}, & \text{if } |s\mathbf{w}| > |\mathbf{w}|, \\
q_s \widetilde{T}^{(q)}_{s\mathbf{w}} + (1 - q_s)\widetilde{T}^{(q)}_{\mathbf{w}}, & \text{if } |s\mathbf{w}| < |\mathbf{w}|.
\end{cases}
\]
\end{definition}

Throughout this paper, we focus on Iwahori--Hecke algebras over the complex numbers with positive real deformation parameters.

Let $(W,S)$ be a Coxeter system of finite rank and define $\mathbb{R}_{>0}^{(W,S)}$ as the subset of $\mathbb{R}_{>0}^{S}$ consisting of tuples $q=(q_s)_{s\in S}$ with $q_s = q_t$ whenever $s$ and $t$ are conjugate. Given $q\in\mathbb{R}_{>0}^{(W,S)}$ and $\mathbf{w}\in W$ with reduced expression $\mathbf{w}=s_{1}\cdots s_{n}$, by \cite[Chapter 17.1]{Davis08} the quantity $q_{\mathbf{w}}:=q_{s_{1}}\cdots q_{s_{n}}$ does not depend on the choice of the expression. Following \cite{Garncarek16} (see also \cite{CaspersKlisseLarsen21,Klisse23-1,Klisse23-2,RaumSkalski23}), we adopt a renormalization of the canonical generators by setting
\[
T^{(q)}_{\mathbf{w}} := q_{\mathbf{w}}^{-1/2} \widetilde{T}^{(q)}_{\mathbf{w}},
\]
which yields the multiplication rule
\begin{equation}
T^{(q)}_s T^{(q)}_{\mathbf{w}} =
\begin{cases}
T^{(q)}_{s\mathbf{w}}, & \text{if } |s\mathbf{w}| > |\mathbf{w}|, \\
T^{(q)}_{s\mathbf{w}} + p_s(q) T^{(q)}_{\mathbf{w}}, & \text{if } |s\mathbf{w}| < |\mathbf{w}|,
\end{cases}
\label{eq:MultiplicationRule}
\end{equation}
for all $s \in S$, $\mathbf{w} \in W$, where $p_s(q) := \frac{q_s - 1}{\sqrt{q_s}}$. The algebra $\mathbb{C}_q[W]$ becomes a $*$-algebra via the involution $(T^{(q)}_{\mathbf{w}})^* := T^{(q)}_{\mathbf{w}^{-1}}$. 

For each $n \in \mathbb{N}$, define the finite-dimensional subspace
\[
\mathbb{C}_q^{(n)}[W] := \operatorname{Span} \left\{ T^{(q)}_{\mathbf{w}} \mid \mathbf{w} \in W,\ |\mathbf{w}| \leq n \right\} \subseteq \mathbb{C}_q[W],
\]
and the associated projection $\chi_{\leq n}^{(q)} \colon \mathbb{C}_q[W] \to \mathbb{C}_q^{(n)}[W]$ by
\[
\chi_{\leq n}^{(q)}(T^{(q)}_{\mathbf{w}}) :=
\begin{cases}
T^{(q)}_{\mathbf{w}}, & |\mathbf{w}| \leq n, \\
0, & |\mathbf{w}| > n.
\end{cases}
\]
For $n \geq 1$ define $\chi_n^{(q)} := \chi_{\leq n}^{(q)} - \chi_{\leq n-1}^{(q)}$, and let $\chi_0^{(q)} := \chi_{\leq 0}^{(q)}$. When no confusion arises, we omit $q$ from the notation.

The proof of the following lemma is straightforward, and therefore omitted.

\begin{lemma} \label{FiltrationLemma}
Let $(W,S)$ be a finite rank Coxeter system and $q \in \mathbb{R}_{>0}^{(W,S)}$. Then the family $(\mathbb{C}_q^{(n)}[W])_{n \in \mathbb{N}}$ defines a \emph{$*$-filtration} of $\mathbb{C}_q[W]$ by finite-dimensional subspaces in the sense of \cite[Subsection V]{Voiculescu90}. That is, the following statements hold:
\begin{enumerate}
    \item $\mathbb{C}_q^{(0)}[W] \subsetneq \mathbb{C}_q^{(1)}[W] \subsetneq \cdots$,
    \item $\mathbb{C}_q^{(0)}[W] = \mathbb{C}1$, $\mathbb{C}_q^{(m)}[W] \cdot \mathbb{C}_q^{(n)}[W] \subseteq \mathbb{C}_q^{(m+n)}[W]$, and $(\mathbb{C}_q^{(n)}[W])^* = \mathbb{C}_q^{(n)}[W]$,
    \item $\mathbb{C}_q[W] = \bigcup_{n \in \mathbb{N}} \mathbb{C}_q^{(n)}[W]$.
\end{enumerate}
\end{lemma}

A well-known result (see, e.g., \cite[Section 2]{Dymara06}) identifies a canonical tracial state on $\mathbb{C}_q[W]$.

\begin{lemma} \label{BasicLemma}
Let $(W,S)$ be a finite rank Coxeter system, let $q \in \mathbb{R}_{>0}^{(W,S)}$, and define $\tau_q := \chi_0^{(q)}$. Then the following statements hold:
\begin{enumerate}
\item For all $\mathbf{v},\mathbf{w}\in W$, 
\begin{equation}
\tau_{q}(T_{\mathbf{v}^{-1}}^{(q)}T_{\mathbf{w}}^{(q)})=\left\{ \begin{array}{cc}
1, & \text{if }\mathbf{v}=\mathbf{w}\\
0, & \text{if }\mathbf{v}\neq\mathbf{w}.
\end{array}\right.\label{eq:TraceMultiplication}
\end{equation}
\item $\tau_{q}$ is a \emph{faithful state}, i.e., $\tau_{q}(1)=1$ and $\tau_{q}(x^{\ast}x)>0$ for every non-zero $x\in\mathbb{C}_{q}[W]$.
\item $\tau_{q}$ is \emph{tracial}, i.e., $\tau_{q}(xy)=\tau_{q}(yx)$ for all $x,y\in\mathbb{C}_{q}[W]$.
\end{enumerate}
\end{lemma}

\begin{proof}
\emph{About (1)}: The first statement is obtained by induction on the word length of $\mathbf{v} \in W$, using the multiplication identity from~\eqref{eq:MultiplicationRule}.

\smallskip

\emph{About (2):} Since $1 = T_e^{(q)}$, it follows immediately that $\tau_q(1) = 1$. Now consider an arbitrary element $x \in \mathbb{C}_q[W]$ of the form $x = \sum_{\mathbf{w} \in W} x(\mathbf{w}) T_{\mathbf{w}}^{(q)}$,
with complex coefficients $x(\mathbf{w}) \in \mathbb{C}$. By the identity in~\eqref{eq:TraceMultiplication}, we compute
\[
\tau_q(x^* x) = \sum_{\mathbf{v}, \mathbf{w} \in W} \overline{x(\mathbf{v})} x(\mathbf{w}) \, \tau_q(T_{\mathbf{v}^{-1}}^{(q)} T_{\mathbf{w}}^{(q)}) = \sum_{\mathbf{w} \in W} |x(\mathbf{w})|^2,
\]
which is strictly positive whenever $x \neq 0$. This proves that $\tau_q$ is a faithful state.

\smallskip

\emph{About (3):} Again using~\eqref{eq:TraceMultiplication}, let $x, y \in \mathbb{C}_q[W]$ be given by $x = \sum_{\mathbf{w} \in W} x(\mathbf{w}) T_{\mathbf{w}}^{(q)}$ and $y = \sum_{\mathbf{w} \in W} y(\mathbf{w}) T_{\mathbf{w}}^{(q)}$. Then we compute
\begin{eqnarray}
\nonumber
\tau_{q}(xy) &=& \sum_{\mathbf{v},\mathbf{w}\in W}x(\mathbf{v})y(\mathbf{w})\tau_{q}(T_{\mathbf{v}}^{(q)}T_{\mathbf{w}}^{(q)}) \\
\nonumber
&=& \sum_{\mathbf{w}\in W}x(\mathbf{w}^{-1})y(\mathbf{w}) \\
\nonumber
&=& \sum_{\mathbf{v},\mathbf{w}\in W}y(\mathbf{v})x(\mathbf{w})\tau_{q}(T_{\mathbf{v}}^{(q)}T_{\mathbf{w}}^{(q)}) \\
&=& \tau_{q}(yx),
\end{eqnarray}
since the sum is symmetric under relabeling. Hence, $\tau_q(xy) = \tau_q(yx)$, establishing the traciality of $\tau_q$.
\end{proof}

Let $(W, S)$ be a Coxeter system of finite rank, and let $q \in \mathbb{R}_{>0}^{(W, S)}$ be a multi-parameter. For notational convenience, we denote, for any $x \in \mathbb{C}_{q}[W]$ and $\mathbf{w} \in W$, $x(\mathbf{w}) := \tau_q(T_{\mathbf{w}^{-1}}^{(q)} x) \in \mathbb{C}$,
so that, by Lemma~\ref{BasicLemma}, each element $x \in \mathbb{C}_{q}[W]$ admits the expansion $x = \sum_{\mathbf{w} \in W} x(\mathbf{w}) T_{\mathbf{w}}^{(q)}$.

Lemma~\ref{FiltrationLemma} places us within the framework of~\cite[Section 1]{OzawaRieffel05}. The faithful tracial state $\tau_q$ allows to construct the associated GNS-Hilbert space $L^{2}(\mathbb{C}_{q}[W], \tau_q)$, which canonically identifies with the Hilbert space $\ell^{2}(W)$ of square-summable complex-valued functions on $W$. The corresponding GNS-representation is faithful so that $\mathbb{C}_q[W]$ can be regarded as a $*$-subalgebra of $\mathcal{B}(\ell^2(W))$. The corresponding action is given by
\[
T^{(q)}_s \delta_{\mathbf{w}} =
\begin{cases}
\delta_{s\mathbf{w}}, & |s\mathbf{w}| > |\mathbf{w}|, \\
\delta_{s\mathbf{w}} + p_s(q)\delta_{\mathbf{w}}, & |s\mathbf{w}| < |\mathbf{w}|.
\end{cases}
\]
for $s\in S$, $\mathbf{w}\in W$, where $(\delta_{\mathbf{w}})_{\mathbf{w}\in W}\subseteq\ell^{2}(W)$ is the canonical orthonormal basis of $\ell^{2}(W)$.

\begin{definition}
Let $(W,S)$ be a finite rank Coxeter system and $q \in \mathbb{R}_{>0}^{(W,S)}$. The corresponding \emph{Hecke C$^*$-algebra} is defined as the norm closure
\[
C^*_{r,q}(W) := \overline{\mathbb{C}_q[W]}^{\|\cdot\|} \subseteq \mathcal{B}(\ell^2(W)).
\]
\end{definition}

For $q = 1$, the operator $T^{(1)}_{\mathbf{w}}$ coincides with the left regular representation operator associated with $\mathbf{w} \in W$, so that $\mathbb{C}_1[W] = \mathbb{C}[W]$ and $C^{\ast}_{r,1}(W) \cong C^*_r(W)$, where $\mathbb{C}[W]$ and $C_r^\ast(W)$ respectively are the group algebra and the reduced group C$^\ast$-algebra of $W$. Thus, Hecke C$^*$-algebras can be viewed as deformations of reduced group C$^*$-algebras of Coxeter groups.\\

Following the framework in \cite{OzawaRieffel05}, define for each $n \in \mathbb{N}$ the orthogonal projection $P_n \in \mathcal{B}(\ell^2(W))$ onto the subspace spanned by $\{\delta_{\mathbf{w}} \mid |\mathbf{w}| = n\}$, and let
\[
D_S := \sum_{n \in \mathbb{N}} n P_n
\]
be the induced unbounded operator on $\ell^{2}(W)$. Then the tuple $(\mathbb{C}_q[W], \ell^2(W), D_S)$ forms a \emph{spectral triple} on $C^*_{r,q}(W)$ in the sense of Connes \cite{Connes89,Connes94}, i.e., $D_{S}$ is a densely defined self-adjoint operator for which $(1+D_{S}^{2})^{-\frac{1}{2}}$ is compact, whose domain is invariant under multiplication with elements in $\mathbb{C}_{q}[W]$, and such that the commutator $[D_{S},x]=D_{S}x-xD_{S}$ is bounded for every $x\in\mathbb{C}_{q}[W]$.

We define a seminorm $L_S^{(q)}$ on $\mathfrak{sa}(C^*_{r,q}(W))$ by
\begin{equation}
L_S^{(q)}(x) :=
\begin{cases}
\|[D_S, x]\|, & \text{if } x \in \mathbb{C}_q[W] \cap \mathfrak{sa}(C^*_{r,q}(W)), \\
\infty, & \text{otherwise}.
\end{cases}
\label{eq:SeminormDefinition}
\end{equation}
For $x \in \mathbb{C}_q[W]$, we find
\begin{eqnarray}
\|[D_S, x] \delta_e\|_2^2 = \sum_{\mathbf{w} \in W} |\mathbf{w}|^2 |x(\mathbf{w})|^2,\label{eq:2-Norm}
\end{eqnarray}
so $L_S^{(q)}(x) = 0$ if and only if $x \in \mathbb{R}1$, and the domain $\operatorname{dom}(L_S^{(q)}) = \mathbb{C}_q[W] \cap \mathfrak{sa}(C^*_{r,q}(W))$ is dense. Thus, the first two conditions of Definition~\ref{CQMSDefinition} are satisfied. This leads us to the following central question of this section:

\begin{question} \label{MainQuestion}
Let $(W,S)$ be a Coxeter system and $q \in \mathbb{R}_{>0}^{(W,S)}$. Under what conditions does the pair $(C^*_{r,q}(W), L_S^{(q)})$ define a compact quantum metric space? That is, when does the associated Monge--Kantorovich metric
\[
d_S(\varphi, \psi) := \sup\left\{ |\varphi(x) - \psi(x)| \,\middle|\, x \in \mathbb{C}_q[W] \cap \mathfrak{sa}(C^*_{r,q}(W)),\, \|[D_S, x]\| \leq 1 \right\}
\]
metrize the weak$^*$-topology on $\mathcal{S}(C^*_{r,q}(W))$?
\end{question}

When the seminorm $L_S^{(q)}$ defined in \eqref{eq:SeminormDefinition} is a Lip-norm, the spectral triple $(\mathbb{C}_q[W], \ell^2(W), D_S)$ is referred to as a \emph{spectral metric space} or \emph{metric spectral triple}.

\vspace{3mm}


\subsection{The Haagerup-type Condition for the Filtration \texorpdfstring{$\mathbb{C}_{q}^{(n)}[W]$}{Cq⁽ⁿ⁾[W]}}

Inspired by Haagerup’s seminal work on reduced group $\mathrm{C}^{\ast}$-algebras of free groups \cite{Haagerup78}, Ozawa and Rieffel introduced in \cite{OzawaRieffel05} the \emph{Haagerup-type condition} in the setting of filtered $\mathrm{C}^{\ast}$-algebras. They showed that this condition ensures that the Monge–Kantorovich metric induced by the associated Dirac operator metrizes the weak$^{\ast}$-topology on the state space. Furthermore, based on the work of Connes in \cite{Connes94}, they verified that this condition is satisfied for natural filtrations of group $\mathrm{C}^{\ast}$-algebras associated with word-hyperbolic groups. The criterion has since been employed in various contexts, e.g., \cite{AguilarHartglassPenneys22}, \cite{ToyotaYang23}, and \cite{AustadKyed25}.

The objective of this subsection is to establish the following characterization of the Haagerup-type condition for the filtration introduced in Lemma~\ref{FiltrationLemma}, specialized to the right-angled setting.

\begin{theorem} \label{HaagerupTypeCondition}
Let $(W, S)$ be a finite rank, right-angled Coxeter system, and let $\Gamma$ denote the finite, undirected, simplicial graph with vertex set $V\Gamma := S$ and edge set $E\Gamma := \{(s, t) \in S \times S \mid m_{s,t} = 2\}$. Then the following two statements hold:
\begin{enumerate}
    \item If $\Gamma$ contains no induced square, then there exists a constant $K > 0$ such that
\[
\|P_i x P_j\| \leq K C_q \|x \delta_e\|_2
\]
for all $q \in \mathbb{R}_{>0}^{(W, S)}$, $i, j, n \in \mathbb{N}$, $x \in \chi_n(\mathbb{C}_q[W])$, where $C_q := \max_{\Gamma_0 \in \mathrm{Cliq}(\Gamma)} \left( \prod_{t \in V \Gamma_0} |p_t(q)| \right)$. \label{FirstStatement}
    
    \item If $\Gamma$ contains an induced square and $q \in \mathbb{R}_{>0}^{(W, S)}$, then there does not exist any constant $K > 0$ such that $\|P_i x P_j\| \leq K \|x \delta_e\|_2$ holds for all $i, j, n \in \mathbb{N}$ and $x \in \chi_n(\mathbb{C}_q[W])$.
\end{enumerate}
\end{theorem}

The proof of Theorem~\ref{HaagerupTypeCondition} relies on the following combinatorial estimate.

Given a clique $\Gamma_0 \in \mathrm{Cliq}(\Gamma)$, define $V(\Gamma_0) := \prod_{s \in \Gamma_0} s \in W$.

\begin{lemma} \label{SetSize}
Let $(W, S)$ be a finite rank, right-angled Coxeter system, and suppose that the graph $\Gamma$ defined in Theorem~\ref{HaagerupTypeCondition} contains no induced square. For $i \in \mathbb{N}$ and $\mathbf{x}, \mathbf{y} \in W$, define $R_{\mathbf{x}, \mathbf{y}}(i)$ to be the number of tuples $(\mathbf{u}, \Gamma_1, \mathbf{v}_1, \mathbf{v}_2, \Gamma_2, \mathbf{w}_1, \mathbf{w}_2)$ where $\mathbf{u}, \mathbf{v}_1, \mathbf{v}_2, \mathbf{w}_1, \mathbf{w}_2 \in W$ and $\Gamma_1, \Gamma_2 \in \mathrm{Cliq}(\Gamma)$ satisfy the conditions:
\[
\mathbf{x} = \mathbf{v}_1 V(\Gamma_1) \mathbf{v}_2, \quad \mathbf{y} = (\mathbf{w}_1 \mathbf{w}_2)^{-1} \mathbf{u}, \quad \mathbf{v}_1, \mathbf{w}_1 \leq \mathbf{u},
\]
and
\[
\begin{aligned}
|\mathbf{u}| &= i, \\
|\mathbf{v}_1| &= \frac{i + |\mathbf{x}| - |\mathbf{y}| - \# V\Gamma_1}{2}, \\
|\mathbf{v}_2| &= |\mathbf{x}| - \frac{i + |\mathbf{x}| - |\mathbf{y}| + \# V\Gamma_1}{2}, \\
|\mathbf{w}_1| &= \frac{i + |\mathbf{x}| - |\mathbf{y}| - \# V\Gamma_2}{2}, \\
|\mathbf{w}_2| &= |\mathbf{x}| - \frac{i + |\mathbf{x}| - |\mathbf{y}| + \# V\Gamma_2}{2}.
\end{aligned}
\]
Then there exists a constant $K > 0$ such that $R_{\mathbf{x}, \mathbf{y}}(i) \leq K$ for all $i \in \mathbb{N}$ and $\mathbf{x}, \mathbf{y} \in W$.
\end{lemma}

\begin{proof}
Fix $i \in \mathbb{N}$ and $\mathbf{x}, \mathbf{y} \in W$. Consider $(\mathbf{u}, \Gamma_1, \mathbf{v}_1, \mathbf{v}_2, \Gamma_2, \mathbf{w}_1, \mathbf{w}_2)$ and $(\mathbf{u}', \Gamma_1', \mathbf{v}_1', \mathbf{v}_2', \Gamma_2', \mathbf{w}_1', \mathbf{w}_2')$ satisfying the above conditions. Since $\Gamma$ contains no induced square, Theorem~\ref{MoussongTheorem} implies that $W$ is word-hyperbolic. Therefore, there exists $\delta \geq 0$ such that for all $\mathbf{g}_1, \mathbf{g}_2, \mathbf{g}_3, \mathbf{g}_4 \in W$,
\[
|\mathbf{g}_1^{-1} \mathbf{g}_2| + |\mathbf{g}_3^{-1} \mathbf{g}_4| \leq \max \left\{ |\mathbf{g}_1^{-1} \mathbf{g}_3| + |\mathbf{g}_2^{-1} \mathbf{g}_4|,\; |\mathbf{g}_1^{-1} \mathbf{g}_4| + |\mathbf{g}_2^{-1} \mathbf{g}_3| \right\} + \delta.
\]

Applying this inequality yields the bound
\[
|\mathbf{x}| + |\mathbf{v}_1^{-1} \mathbf{v}_1'| \leq \max \left\{ |\mathbf{v}_1| + |\mathbf{x}^{-1} \mathbf{v}_1'|,\; |\mathbf{v}_1'| + |\mathbf{x}^{-1} \mathbf{v}_1| \right\} + \delta.
\]
From the word length constraints, we have $\mathbf{v}_1, \mathbf{v}_1' \leq \mathbf{x}$, implying
\[
|\mathbf{v}_1^{-1} \mathbf{v}_1'| \leq \frac{|\# V\Gamma_1 - \# V\Gamma_1'|}{2} + \delta \leq \# S + \delta.
\]
Analogously, $|\mathbf{v}_{2}^{-1}\mathbf{v}_{2}^{\prime}|\leq\#S+\delta$, $|\mathbf{w}_{2}^{\prime}\mathbf{w}_{2}^{-1}|\leq\#S+\delta$, and $|\mathbf{v}_{1}^{-1}\mathbf{w}_{1}|\leq\#S+\delta$.

Let $R$ denote the cardinality of the ball of radius $\#S + \delta$ in $W$. Using $\mathbf{u} = \mathbf{w}_1 \mathbf{w}_2 \mathbf{y}$, we conclude that
\[
R_{\mathbf{x}, \mathbf{y}}(i) \leq \left( \# \mathrm{Cliq}(\Gamma) \right)^2 R^4,
\]
which establishes the claim.
\end{proof}

Beyond Lemma \ref{SetSize}, the proof of Theorem \ref{HaagerupTypeCondition} relies on a suitable decomposition of the canonical basis elements in the Iwahori–Hecke algebra. This decomposition was previously employed in \cite{CaspersKlisseLarsen21} to establish Khintchine-type and Haagerup-type inequalities for right-angled Hecke C$^{\ast}$-algebras.

Let $(W, S)$ be a finite rank, right-angled Coxeter system, and let $\Gamma$ be the graph appearing in Theorem \ref{HaagerupTypeCondition}. For each $\mathbf{w} \in W$, fix a reduced expression $\mathbf{w} = s_1 \cdots s_n$ with $s_1, \ldots, s_n \in S$, and denote by $\mathcal{I}$ the collection of all such reduced expressions.

\begin{definition}[{\cite[Definition 2.3]{CaspersKlisseLarsen21}}] \label{PermutationDefinition}
Let $(W, S)$ be a finite rank, right-angled Coxeter system, and let $\mathbf{w} \in W$ admit a reduced expression $\mathbf{w} = s_1 \cdots s_n$ with $s_1, \ldots, s_n \in S$. Given integers $0 \leq l \leq n$, $0 \leq k \leq n - l$, a clique $\Gamma_0 \in \mathrm{Cliq}(\Gamma, l)$, and a pair $(\Gamma_1, \Gamma_2) \in \mathrm{Comm}(\Gamma_0)$, define (when it exists) $\sigma := \sigma_{l, k, \Gamma_0, \Gamma_1, \Gamma_2}$ to be the unique permutation of $\{1, \ldots, n\}$  satisfying:
\begin{enumerate}
    \item $s_1 \cdots s_n = s_{\sigma(1)} \cdots s_{\sigma(n)}$;
    \item $s_{\sigma(k+1)} \cdots s_{\sigma(k+l)} \in \mathcal{I}$ and $\{s_{\sigma(k+1)}, \ldots, s_{\sigma(k+l)}\} = V\Gamma_0$;
    \item $s_{\sigma(1)} \cdots s_{\sigma(k)} \in \mathcal{I}$ with $\vert s_{\sigma(1)} \cdots s_{\sigma(k)} s \vert = k-1$ for all $s \in V\Gamma_1$;
    \item $\vert s_{\sigma(1)} \cdots s_{\sigma(k)} s \vert = k+1$ for all $s \in V\mathrm{Link}(\Gamma_0) \setminus V\Gamma_1$;
    \item $s_{\sigma(k+l+1)} \cdots s_{\sigma(n)} \in \mathcal{I}$ with $\vert s s_{\sigma(k+l+1)} \cdots s_{\sigma(n)} \vert = n - k - l - 1$ for all $s \in V\Gamma_2$;
    \item $\vert s s_{\sigma(k+l+1)} \cdots s_{\sigma(n)} \vert = n - k - l + 1$ for all $s \in V\mathrm{Link}(\Gamma_0) \setminus V\Gamma_2$;
    \item $\sigma(i) < \sigma(j)$ whenever $i < j$ and $s_i = s_j$.
\end{enumerate}
When $l = 0$, we set $\Gamma_0 = \emptyset$, so condition (ii) is vacuously satisfied.
\end{definition}

The following proposition describes how each basis element $T_{\mathbf{w}}^{(q)}$ may be expressed as a sum of products involving creation, diagonal, and annihilation operators.

\begin{proposition}[{\cite[Proposition 2.6]{CaspersKlisseLarsen21}}] \label{GeneratorDecomposition}
Let $(W, S)$ be a right-angled Coxeter system, and let $\mathbf{w} \in W$ have a fixed reduced expression $\mathbf{w}=s_1 \cdots s_n$. For $\mathbf{u} \in W$, define $Q_{\mathbf{u}} \in \mathcal{B}(\ell^2(W))$ to be the orthogonal projection onto the closed subspace
\[
\overline{\mathrm{Span}}^{\|\cdot\|_2} \{ \delta_{\mathbf{v}} \mid \mathbf{v} \in W,\, \mathbf{u} \leq \mathbf{v} \} \subseteq \ell^2(W).
\]
Then the operator $T_{\mathbf{w}}^{(q)} \in \mathcal{B}(\ell^2(W))$ admits a decomposition as a finite sum of the form
\begin{equation*}
\begin{split}
T_{\mathbf{w}}^{(q)} = \sum_{l=0}^n \sum_{k=0}^{n-l} \sum_{\Gamma_0 \in \mathrm{Cliq}(\Gamma, l)} \sum_{(\Gamma_1, \Gamma_2) \in \mathrm{Comm}(\Gamma_0)} & (Q_{s_{\sigma(1)}} T_{s_{\sigma(1)}}^{(1)} Q_{s_{\sigma(1)}}^{\perp}) \cdots (Q_{s_{\sigma(k)}} T_{s_{\sigma(k)}}^{(1)} Q_{s_{\sigma(k)}}^{\perp}) \\
& \times \left( \prod_{t \in V \Gamma_0} p_t(q) \right) Q_{V(\Gamma_0)} \\
& \times (Q_{s_{\sigma(k+l+1)}}^{\perp} T_{s_{\sigma(k+l+1)}}^{(1)} Q_{s_{\sigma(k+l+1)}}) \cdots (Q_{s_{\sigma(n)}}^{\perp} T_{s_{\sigma(n)}}^{(1)} Q_{s_{\sigma(n)}}),
\end{split}
\end{equation*}
where $\sigma = \sigma_{l, k, \Gamma_0, \Gamma_1, \Gamma_2}$ is as in Definition \ref{PermutationDefinition}. If no such $\sigma$ exists for given parameters, the corresponding summand is taken to be zero.
\end{proposition}

\begin{remark} \label{OperatorRemark}
Let $(W, S)$ be a right-angled Coxeter system and $q = (q_s)_{s \in S} \in \mathbb{R}_{>0}^{(W, S)}$. For each $s \in S$, the operators
\[
Q_s T_s^{(1)} Q_s^{\perp} = Q_s T_s^{(q)} Q_s^{\perp}, \quad
Q_s^{\perp} T_s^{(1)} Q_s = Q_s^{\perp} T_s^{(q)} Q_s, \quad
p_s(q) Q_s = Q_s T_s^{(q)} Q_s
\]
are referred to as the \emph{creation}, \emph{annihilation}, and \emph{diagonal} operators associated with $T_s^{(q)}$, respectively, as considered in \cite{CaspersKlisseLarsen21} and \cite{Klisse25}. The terminology is justified by the following identities for $\mathbf{w} \in W$:
\[
Q_s T_s^{(1)} Q_s^{\perp} \delta_{\mathbf{w}} = 
\begin{cases}
\delta_{s\mathbf{w}}, & \text{if } s \nleq \mathbf{w}, \\
0, & \text{if } s \leq \mathbf{w},
\end{cases}
\quad
Q_s^{\perp} T_s^{(1)} Q_s \delta_{\mathbf{w}} =
\begin{cases}
\delta_{s\mathbf{w}}, & \text{if } s \leq \mathbf{w}, \\
0, & \text{if } s \nleq \mathbf{w},
\end{cases}
\]
\[
Q_s \delta_{\mathbf{w}} =
\begin{cases}
\delta_{\mathbf{w}}, & \text{if } s \leq \mathbf{w}, \\
0, & \text{if } s \nleq \mathbf{w}.
\end{cases}
\]
These formulas play an important role in the proof of Theorem \ref{HaagerupTypeCondition}.
\end{remark}

We are now in a position to prove Theorem~\ref{HaagerupTypeCondition}. The proof follows the general approach developed in~\cite[Sections 4 and 5]{OzawaRieffel05}. A similar general approach will be employed in the proof of Lemma \ref{HaagerupTypeCondition2}.

\begin{proof}[Proof of Theorem~\ref{HaagerupTypeCondition}]
\emph{About (1)}: Suppose that the graph $\Gamma$ contains no induced squares, and let $K > 0$ be the constant provided by Lemma~\ref{SetSize}. Fix $q \in \mathbb{R}_{>0}^{(W, S)}$, $i, j, n \in \mathbb{N}$, let $x \in \chi_n(\mathbb{C}_q[W])$, and take $\xi \in P_j \ell^2(W)$. By Proposition~\ref{GeneratorDecomposition} and Remark~\ref{OperatorRemark}, we have
\begin{eqnarray*}
P_{i}x\xi = \sum_{(l,k,\Gamma_{0},\Gamma_{1},\Gamma_{2})}\sum_{\mathbf{w}\in W:|\mathbf{w}|=n}\left(\prod_{t\in V\Gamma_{0}}p_{t}(q)\right) \left(\sum\left\{ \left.x(\mathbf{w})\xi(\mathbf{v})\delta_{\mathbf{w}_{l,\Gamma_{0},\Gamma_{1},\Gamma_{2}}^{(1)}\mathbf{w}_{l,\Gamma_{0},\Gamma_{1},\Gamma_{2}}^{(3)}\mathbf{v}} \,  \right|\right.\right. \qquad \qquad & \\
\mathbf{v}\in W,\left|\mathbf{v}\right|=j,(\mathbf{w}_{l,\Gamma_{0},\Gamma_{1},\Gamma_{2}}^{(3)})^{-1}\leq\mathbf{v},\mathbf{w}_{l,\Gamma_{0},\Gamma_{1},\Gamma_{2}}^{(2)}\leq\mathbf{w}_{l,\Gamma_{0},\Gamma_{1},\Gamma_{2}}^{(3)}\mathbf{v}, \qquad \quad &  \\
\left.\left.\mathbf{w}_{l,\Gamma_{0},\Gamma_{1},\Gamma_{2}}^{(1)}\leq\mathbf{w}_{l,\Gamma_{0},\Gamma_{1},\Gamma_{2}}^{(1)}\mathbf{w}_{l,\Gamma_{0},\Gamma_{1},\Gamma_{2}}^{(3)}\mathbf{v}\right\} \right), &
\end{eqnarray*}
where the outer sum runs over all tuples $(l, \Gamma_0, \Gamma_1, \Gamma_2)$ with $0 \leq l \leq n$, $\Gamma_0 \in \mathrm{Cliq}(\Gamma, l)$, and $(\Gamma_1, \Gamma_2) \in \mathrm{Comm}(\Gamma_0)$, and where for each such tuple and each $\mathbf{w} \in W$ with $|\mathbf{w}| = n$ (for which the associated permutation $\sigma$ from Definition~\ref{PermutationDefinition} exists), we define
\[
\mathbf{w}_{l,\Gamma_{0},\Gamma_{1},\Gamma_{2}}^{(1)}:=s_{\sigma(1)}...s_{\sigma(k_{0})},\; \mathbf{w}_{l,\Gamma_{0},\Gamma_{1},\Gamma_{2}}^{(2)}:=s_{\sigma(k_{0}+1)}...s_{\sigma(k_{0}+l)},\; \mathbf{w}_{l,\Gamma_{0},\Gamma_{1},\Gamma_{2}}^{(3)}:=s_{\sigma(k_{0}+l+1)} \cdots s_{\sigma(n)},
\]
with $k_0 := \frac{1}{2}(i - j + n - l)$. (Note that $k_0$ depends on the tuple $(l, \Gamma_0, \Gamma_1, \Gamma_2)$ and $\mathbf{w}$, though we suppress this in the notation for brevity.)

It follows that
\begin{eqnarray*}
& & (P_{i}x\xi)(\mathbf{v})\\
&=& \sum_{(l,\Gamma_{0},\Gamma_{1},\Gamma_{2})}\left(\prod_{t\in V \Gamma_{0}}p_{t}(q)\right) \\
&\times& \left(\sum\left\{ \left.x(\mathbf{w}_{l,\Gamma_{0},\Gamma_{1},\Gamma_{2}}^{(1)}\mathbf{w}_{l,\Gamma_{0},\Gamma_{1},\Gamma_{2}}^{(2)}\mathbf{w}_{l,\Gamma_{0},\Gamma_{1},\Gamma_{2}}^{(3)})\xi((\mathbf{w}_{l,\Gamma_{0},\Gamma_{1},\Gamma_{2}}^{(1)}\mathbf{w}_{l,\Gamma_{0},\Gamma_{1},\Gamma_{2}}^{(3)})^{-1}\mathbf{v})\right|\right.\right.\\
& & \qquad \qquad \qquad \qquad \mathbf{w}\in W, \left|\mathbf{w}\right|=n,(\mathbf{w}_{l,\Gamma_{0},\Gamma_{1},\Gamma_{2}}^{(3)})^{-1}\leq(\mathbf{w}_{l,\Gamma_{0},\Gamma_{1},\Gamma_{2}}^{(1)}\mathbf{w}_{l,\Gamma_{0},\Gamma_{1},\Gamma_{2}}^{(3)})^{-1}\mathbf{v}\\
& & \qquad \qquad  \qquad \qquad \qquad \qquad \qquad \qquad \left.\left.\mathbf{w}_{l,\Gamma_{0},\Gamma_{1},\Gamma_{2}}^{(2)}\leq(\mathbf{w}_{l,\Gamma_{0},\Gamma_{1},\Gamma_{2}}^{(1)})^{-1}\mathbf{v},\mathbf{w}_{l,\Gamma_{0},\Gamma_{1},\Gamma_{2}}^{(1)}\leq\mathbf{v}\right\} \right),
\end{eqnarray*}
for $\mathbf{v}\in W$ with $|\mathbf{v}|=i$ and hence,
\begin{equation} \label{eq:TupleDefinition}
\begin{split}
& \left|(P_{i}x\xi)(\mathbf{v})\right| \\
&\leq (\#\text{Cliq}(\Gamma))^{2}C_{q}\sum_{\Gamma_{0}\in\text{Cliq}(\Gamma)}\left(\sum\left\{ \left.|x(\mathbf{u}V(\Gamma_{0})\mathbf{u}^{\prime})||\xi((\mathbf{u}\mathbf{u}^{\prime})^{-1}\mathbf{v})| \: \right|\right. \, \right. \mathbf{u},\mathbf{u}^{\prime}\in W \text{ with }\\
&  \qquad  \qquad \qquad \qquad  \qquad \left|\mathbf{u}\right|=2^{-1}(i-j+n-\#V\Gamma_{0}),\left|\mathbf{u}^{\prime}\right|=n-2^{-1}(i-j+n+\#V\Gamma_{0}), \\
&  \qquad \qquad \qquad \qquad \qquad \qquad  \qquad \qquad \qquad  \qquad \qquad \left.\left.(\mathbf{u}^{\prime})^{-1}\leq(\mathbf{u}\mathbf{u}^{\prime})^{-1}\mathbf{v},V(\Gamma_{0})\leq\mathbf{u}^{-1}\mathbf{v},\mathbf{u}\leq\mathbf{v}\right\} \right), 
\end{split}
\end{equation}
where $C_q := \max_{\Gamma_0 \in \mathrm{Cliq}(\Gamma)} \prod_{t \in V \Gamma_0} |p_t(q)|$. Applying the Cauchy–Schwarz inequality gives
\[
\| P_i x \xi \|_2^2 \leq (\# \mathrm{Cliq}(\Gamma))^4 C_q^2 \sum_{\mathbf{v}\in W : |\mathbf{v}| = i}
\left( \sum_{(\Gamma_0, \mathbf{u}, \mathbf{u}')} |x(\mathbf{u} V(\Gamma_0) \mathbf{u}')|^2 \right)
\left( \sum_{(\Gamma_0, \mathbf{u}, \mathbf{u}')} |\xi((\mathbf{u} \mathbf{u}')^{-1} \mathbf{v})|^2 \right),
\]
where the sums run over tuples as in \eqref{eq:TupleDefinition}. We may reindex these sums to estimate
\[
\| P_i x \xi \|_2^2 \leq (\# \mathrm{Cliq}(\Gamma))^4 C_q^2 \sum_{\mathbf{x} \in W : |\mathbf{x}| = n} \sum_{\mathbf{y} \in W : |\mathbf{y}| = j} R_{\mathbf{x}, \mathbf{y}}(i) |x(\mathbf{x})|^2\, |\xi(\mathbf{y})|^2,
\]
where $R_{\mathbf{x}, \mathbf{y}}(i)$ is as in Lemma~\ref{SetSize}. Thus, applying Lemma~\ref{SetSize}, we obtain
\[
\| P_i x \xi \|_2 \leq \sqrt{K} (\# \mathrm{Cliq}(\Gamma))^2 C_q \|x \delta_e\|_2 \|\xi\|_2,
\]
as required.\\

\emph{About (2)}: Now assume that $\Gamma$ contains an induced square. Then there exist $u, v, s, t \in S$ such that
\[
m_{u,s} = m_{u,t} = m_{v,s} = m_{v,t} = 2 \quad \text{and} \quad m_{u,v} = m_{s,t} = \infty.
\]
For each $n \in \mathbb{N}_{\geq 1}$, define
\[
x_n := \frac{1}{n} \sum_{i=1}^n T_{(uv)^i (st)^{n - i}}^{(q)} \in \chi_{2n}(\mathbb{C}_q[W]), \quad
\xi_n := \frac{1}{\sqrt{2n}} \sum_{j=1}^{2n} \delta_{(uv)^j (st)^{2n - j}} \in P_{4n} \ell^2(W).
\]
Then clearly $\| x_n \delta_e \|_2 = n^{-1/2}$ and $\| \xi_n \|_2 = 1$. Observe that
\begin{eqnarray*}
\| P_{6n} x_n P_{4n} \| &\geq& \| P_{6n} x_n \xi_n \|_2 = \frac{1}{\sqrt{2n^3}} \left\| \sum_{i=1}^n \sum_{j=1}^{2n} \delta_{(uv)^{i + j} (st)^{3n - (i + j)}} \right\|_2 \\
\nonumber 
&\geq& \frac{1}{\sqrt{2n^3}} \left\| \sum_{i=n+1}^{2n} n \delta_{(uv)^i (st)^{3n - i}} \right\|_2 = \frac{1}{\sqrt{2}},
\end{eqnarray*}
and hence
\[
\sup_{i,j \in \mathbb{N}} \frac{\| P_i x_n P_j \|}{\| x_n \delta_e \|_2} \geq \sqrt{\frac{n}{2}} \to \infty .
\]
It follows that there is no constant $C > 0$ such that $\| P_i x P_j \| \leq C \| x \delta_e \|_2$ for all $i, j, n \in \mathbb{N}$ and all $x \in \chi_n(\mathbb{C}_q[W])$, as claimed.
\end{proof}

\vspace{3mm}


\subsection{Quantum Metric Structures on Right-Angled Hecke C$^\ast$-Algebras}

Let $(W,S)$ be a finite rank, right-angled Coxeter system, and assume that the graph $\Gamma$ defined in Theorem~\ref{HaagerupTypeCondition} contains no induced square. The results of \cite[Sections 2 and 3]{OzawaRieffel05} imply that, for every $q \in \mathbb{R}_{>0}^{(W,S)}$, the set $\mathcal{L}_1(C_{r,q}^{*}(W), L_S^{(q)}, \tau_q)$ is totally bounded with respect to the norm topology. Consequently, by Theorem~\ref{CQMSCharacterization}, the pair $(C_{r,q}^{*}(W),L_S^{(q)})$ forms a compact quantum metric space,  partially answering Question \ref{MainQuestion}.

\begin{theorem} \label{CQMSStatement}
Let $(W,S)$ be a finite rank, right-angled Coxeter system, and suppose that the graph $\Gamma$ defined in Theorem~\ref{HaagerupTypeCondition} contains no induced square. Then $(C_{r,q}^{*}(W), L_S^{(q)})$ is a compact quantum metric space.
\end{theorem}

In this subsection, we prove a slightly stronger result that will be instrumental in Section~\ref{qGHContinuity}.

For notational convenience, we write $\mathcal{L}_1^{(q)} := \mathcal{L}_1(C_{r,q}^{*}(W), L_S^{(q)})$ and 
$\mathcal{B}_1^{(q)} := \mathcal{L}_1(C_{r,q}^{*}(W), L_S^{(q)}, \tau_q)$ for $\quad q \in \mathbb{R}_{>0}^{(W,S)}$. Explicitly,
\[
\mathcal{L}_1^{(q)} = \{ x \in \mathbb{C}_q[W] \mid L_S^{(q)}(x) \leq 1 \}, \quad
\mathcal{B}_1^{(q)} = \{ x \in \mathbb{C}_q[W] \mid \tau_q(x) = 0,\ L_S^{(q)}(x) \leq 1 \}.
\]

\begin{proposition} \label{UnityBoundedness}
Let $(W,S)$ be a finite rank, right-angled Coxeter system, and suppose that the graph $\Gamma$ defined in Theorem~\ref{HaagerupTypeCondition} contains no induced square. Furthermore, let $\mathcal{K} \subseteq \mathbb{R}_{>0}^{(W,S)}$ be a compact subset. Then the set
\[
\mathbf{B} := \bigcup_{q \in \mathcal{K}} \mathcal{B}_1^{(q)} \subseteq \mathcal{B}(\ell^2(W))
\]
is totally bounded in the operator norm.
\end{proposition}

The proof of Proposition~\ref{UnityBoundedness} adapts the techniques of \cite[Sections 2 and 3]{OzawaRieffel05}, and in particular uses the following estimate.

\begin{lemma}[{\cite[Section 2]{OzawaRieffel05}}]
Let $(W,S)$ be a finite rank, right-angled Coxeter system and let $q \in \mathbb{R}_{>0}^{(W,S)}$. Then for any $x \in \mathbb{C}_q[W]$ and $N \in \mathbb{N}$,
\[
\left\| \sum_{i,j \in \mathbb{N} : |i-j| > N} P_i x P_j \right\|
\leq 2\pi L_S^{(q)}(x) \sqrt{ \sum_{k \in \mathbb{Z} : |k| > N} \frac{1}{k^2} },
\]
where the sum on the left converges in the strong operator topology.
\end{lemma}

\begin{proof}[Proof of Proposition~\ref{UnityBoundedness}]
Let $\varepsilon > 0$. We aim to cover $\mathbf{B} = \bigcup_{q \in \mathcal{K}} \mathcal{B}_1^{(q)}$ with finitely many $\varepsilon$-balls.

First, choose $N \in \mathbb{N}$ sufficiently large such that $\| \sum_{|i-j| > N} P_i x P_j \| < \frac{\varepsilon}{3}$ for all $x \in \mathbf{B}$.

Now let $M \in \mathbb{N}$. By Theorem~\ref{HaagerupTypeCondition}, the Cauchy--Schwarz inequality, and the identity in \eqref{eq:2-Norm}, we have
\begin{align*}
\left\| \sum_{|i-j| \leq N} P_i \chi_{>M}(x) P_j \right\| 
&= \left\| \sum_{m > M} \sum_{|i| \leq N} \sum_{j \geq \max\{0,i\}} P_j \chi_m(x) P_{j-i} \right\| \\
&\leq K(2N+1) C_q \sum_{m > M} \| \chi_m(x) \delta_e \|_2 \\
&\leq K(2N+1) C_q \left( \sum_{m > M} \frac{1}{m^2} \right)^{1/2}
       \left( \sum_{m > M} m^2 \| \chi_m(x) \delta_e \|_2^2 \right)^{1/2} \\
&\leq K (2N+1) C_q \left( \sum_{m > M} \frac{1}{m^2} \right)^{1/2},
\end{align*}
for all $q \in \mathcal{K}$ and $x \in \mathcal{B}_1^{(q)}$, where $K > 0$ is a constant and $C_q = \max_{\Gamma_0 \in \mathrm{Cliq}(\Gamma)} \left( \prod_{t \in \Gamma_0} |p_t(q)| \right)$. Since $\mathcal{K}$ is compact, we may choose $M$ large enough such that
\[
\left\| \sum_{|i-j| \leq N} P_i \chi_{>M}(x) P_j \right\| < \frac{\varepsilon}{3}
\quad \text{for all } x \in \mathbf{B}.
\]

From Proposition~\ref{GeneratorDecomposition}, there exists $\delta > 0$ such that
\[
\left\| T_{\mathbf{w}}^{(q)} - T_{\mathbf{w}}^{(q')} \right\| 
< \frac{\varepsilon}{6N+3} \left( \sum_{\mathbf{w} \in W \setminus \{e\} : |\mathbf{w}| \leq M} \frac{1}{|\mathbf{w}|^2} \right)^{-1/2}
\]
for all $q, q' \in \mathcal{K}$ with $\|q - q'\|_1 < \delta$ and all $\mathbf{w} \in W$ with $|\mathbf{w}| \leq M$, where $\|\cdot\|_1$ denotes the $\ell^1$-norm on $\mathbb{R}^{S}$. Furthermore, by the compactness of $\mathcal{K}$, we can find finitely many $q_1, \dots, q_n \in \mathcal{K}$ such that $\mathcal{K} \subseteq \bigcup_{i=1}^{n} B_{\delta}(q_i)$, where $B_{\delta}(q_i) := \{ q \in \mathcal{K} \mid \|q - q_i\|_1 < \delta \}$.

We now prove the following claim.\\

\emph{Claim.} For every $q \in \mathcal{K}$, there exists $1 \leq l \leq n$ such that
\[
\left\| \sum_{|i-j| \leq N} P_i \chi_{\leq M}(x - x^{(q_l)}) P_j \right\| < \frac{\varepsilon}{3},
\]
where $x^{(q_l)} := \sum_{\mathbf{w} \in W} x(\mathbf{w}) T_{\mathbf{w}}^{(q_l)} \in \mathbb{C}_{q_l}[W]$.

\emph{Proof of the Claim:}
Let $q \in \mathcal{K}$. Then there exists $1 \leq l \leq n$ such that $\|q - q_l\|_1 < \delta$. By the Cauchy--Schwarz inequality,
\begin{align*}
\left\| \sum_{|i-j| \leq N} P_i \chi_{\leq M}(x - x^{(q_l)}) P_j \right\|
&= \left\| \sum_{|i| \leq N} \sum_{j \geq  \max\{0,i\}} P_j \chi_{\leq M}(x - x^{(q_l)}) P_{j-i} \right\| \\
&\leq \frac{\varepsilon}{6N+3} \left( \sum_{\mathbf{w} \in W \setminus \{e\} : |\mathbf{w}| \leq M} \frac{1}{|\mathbf{w}|^2} \right)^{-1/2}
      \sum_{i \leq N} \sum_{\mathbf{w} \in W \setminus \{e\} : |\mathbf{w}| \leq M} |x(\mathbf{w})| \\
&< \frac{\varepsilon}{3},
\end{align*}
as claimed.\\

Combining the claim with the earlier estimates, we find that for any $x \in \mathbf{B}$, there exists $1 \leq l \leq n$ such that $\left\| x - \sum_{|i-j| \leq N} P_i \chi_{\leq M}(x^{(q_l)}) P_j \right\| < \varepsilon$. Moreover, using again the Cauchy--Schwarz inequality and \eqref{eq:2-Norm}, we estimate
\begin{align*}
\| \chi_{\leq M}(x^{(q_l)}) \delta_e \|_2
&\leq \left( \sum_{m=1}^{M} \frac{1}{m^2} \right)^{1/2} 
      \left( \sum_{m=1}^{M} m^2 \| \chi_m(x^{(q_l)}) \delta_e \|_2^2 \right)^{1/2} \\
&\leq \left( \sum_{m=1}^{M} \frac{1}{m^2} \right)^{1/2} \| [D_S, x] \delta_e \|_2 \\
&\leq \left( \sum_{m=1}^{M} \frac{1}{m^2} \right)^{1/2}.
\end{align*}
Thus, it suffices to show that the image under the map $\mathcal{B}(\ell^2(W)) \ni T \mapsto \sum_{|i-j| \leq N} P_i T P_j$ of the union of the sets
\[
\mathcal{Q}_l := \left\{ y \in \chi_{\leq N}(\mathbb{C}_{q_l}[W]) \,\middle|\, \tau_{q_l}(y) = 0,\ 
\| \chi_{\leq M}(y) \delta_e \|_2^2 \leq \sum_{m=1}^{M} \frac{1}{m^2} \right\}
\]
for $1 \leq l \leq n$ is totally bounded in the operator norm. This follows since the map is linear and each $\mathcal{Q}_l$ is a bounded subset of a finite-dimensional normed space. This completes the proof.
\end{proof}

\vspace{3mm}


\section{Continuity in the Quantum Gromov-Hausdorff propinquity\label{qGHContinuity}} \label{Section3}

\vspace{3mm}

Motivated by developments in high-energy physics, Rieffel introduced in \cite{Rieffel04} the \emph{quantum Gromov--Hausdorff distance} on the class of compact quantum metric spaces. This construction serves as a non-commutative analogue of the classical Gromov--Hausdorff distance, enabling the study of convergence phenomena within an operator-algebraic framework. The quantum Gromov--Hausdorff distance is non-negative, symmetric, and satisfies the triangle inequality. However, a vanishing distance does not necessarily imply the existence of a Lip-norm preserving $\ast$-isomorphism between the associated C$^{\ast}$-algebras. Instead, as established in \cite[Theorem 7.7]{Rieffel04}, two compact quantum metric spaces have zero distance if and only if they are isometric in the sense of \cite[Definition 6.3]{Rieffel04}.

Since Rieffel’s seminal work, several alternative constructions of a quantum analog of the Gromov--Hausdorff distance have emerged (see \cite{Kerr03, Li03, Rieffel04, Li06, KerrHanfeng09, Latremoliere15, Latremoliere16, Latremoliere19, Latremoliere22}), each offering distinct advantages and limitations. In this article, we consider Latrémolière’s \emph{quantum Gromov--Hausdorff propinquity}, introduced in \cite{Latremoliere16}, which refines Rieffel’s notion by enforcing compatibility with algebraic structure through the use of Leibniz seminorms. We briefly review the relevant definitions below.

\begin{definition}[{\cite[Definition 3.6]{Latremoliere16}}]
For unital C$^{\ast}$-algebras $A$ and $B$, a \emph{bridge} $\gamma = (Z, \omega, \pi_A, \pi_B)$ from $A$ to $B$ consists of a unital C$^{\ast}$-algebra $Z$, a self-adjoint element $\omega \in Z$ such that the \emph{1-level set}
  \[
  \mathcal{S}_{1}(\omega) := \left\{ \varphi \in \mathcal{S}(Z) \,\middle|\, \varphi(\omega z) = \varphi(z \omega) = \varphi(z) \ \text{for all } z \in Z \right\}
  \]
  is non-empty, and injective unital $\ast$-homomorphisms $\pi_A: A \hookrightarrow Z$ and $\pi_B: B \hookrightarrow Z$.

The element $\omega$ is called the \emph{pivot} of the bridge. The collection of all bridges from $A$ to $B$ is denoted by $\text{\calligra Bridges}(A \rightarrow B)$.
\end{definition}

Let $(A, L_A)$ and $(B, L_B)$ be Leibniz quantum compact metric spaces, and let $\gamma = (Z, \omega, \pi_A, \pi_B)$ be a bridge from $A$ to $B$. Following \cite[Definitions 3.14--3.17]{Latremoliere16}, we define the associated notions of \emph{reach}, \emph{height}, and \emph{length} as follows:

\begin{enumerate}[label=(\roman*)]
  \item The \emph{reach} of $\gamma$ with respect to $L_A$ and $L_B$ is given by
  \[
  \mathfrak{reach}(\gamma \mid L_A, L_B) := \mathrm{Haus} \left( \pi_A\left(\mathcal{L}_1(A, L_A)\right)\omega, \, \omega \pi_B\left(\mathcal{L}_1(B, L_B)\right) \right),
  \]
  where the Hausdorff distance is taken with respect to the operator norm of $Z$.
  
  \item The \emph{height} of $\gamma$ with respect to $L_A$ and $L_B$ is defined as
\begin{eqnarray}
\nonumber
  \mathfrak{height}(\gamma \mid L_A, L_B) := \max \{ & \mathrm{Haus}\left( \mathcal{S}(A), \mathcal{S}_1(\omega) \circ \pi_A \right),\\
\nonumber
& \mathrm{Haus}\left( \mathcal{S}(B), \mathcal{S}_1(\omega) \circ \pi_B \right) \},
\end{eqnarray}
  where the Hausdorff distances are taken with respect to the Monge--Kantorovich metrics $d_{L_A}$ and $d_{L_B}$.

  \item The \emph{length} of the bridge is then
  \[
  \lambda(\gamma \mid L_A, L_B) := \max\left\{ \mathfrak{reach}(\gamma \mid L_A, L_B), \, \mathfrak{height}(\gamma \mid L_A, L_B) \right\}.
  \]
\end{enumerate}

\begin{definition}[{\cite[Definitions 3.20 and 3.22]{Latremoliere16}}]
Let $(A, L_A)$ and $(B, L_B)$ be Leibniz quantum compact metric spaces. A \emph{trek} $\Gamma$ from $(A, L_A)$ to $(B, L_B)$ is a finite sequence of bridges $\Gamma := (\gamma_i)_{1\leq i \leq n}$, where each $\gamma_i$ is a bridge from a Leibniz quantum compact metric space $(A_i, L_i)$ to a Leibniz quantum compact metric space $(A_{i+1}, L_{i+1})$, with $(A_1, L_1) = (A, L_A)$ and $(A_{n+1}, L_{n+1}) = (B, L_B)$. The set of all such treks is denoted by $\text{\calligra Treks}((A, L_A) \rightarrow (B, L_B))$.

The \emph{length} of the trek $\Gamma$ is defined as
\[
\lambda(\Gamma) := \sum_{i=1}^{n} \lambda(\gamma_i \mid L_i, L_{i+1}).
\]
\end{definition}

We now recall the central notion introduced in \cite{Latremoliere16}.

\begin{definition}[{\cite[Definition 4.2]{Latremoliere16}}]
Let $(A, L_A)$ and $(B, L_B)$ be Leibniz quantum compact metric spaces. The \emph{quantum Gromov--Hausdorff propinquity} between $(A, L_A)$ and $(B, L_B)$ is defined by
\[
\Lambda((A, L_A), (B, L_B)) := \inf \left\{ \lambda(\Gamma) \mid \Gamma \in \text{\calligra Treks}((A, L_A) \rightarrow (B, L_B)) \right\}.
\]
\end{definition}

\begin{remark} \label{EstimateRemark} 
Any bridge $\gamma$ from $A$ to $B$ may be viewed as a trek of length one from $(A, L_A)$ to $(B, L_B)$. Therefore, by definition, the propinquity satisfies
\[
\Lambda((A, L_A), (B, L_B)) \leq \lambda(\gamma \mid L_A, L_B).
\]
\end{remark}

By \cite[Proposition 4.6]{Latremoliere16}, the quantum Gromov--Hausdorff propinquity is finite. Moreover, it is symmetric and satisfies the triangle inequality \cite[Proposition 4.7]{Latremoliere16}. Most importantly, distance zero characterizes quantum isometry: if $\Lambda((A, L_A), (B, L_B)) = 0$, then there exists a $\ast$-isomorphism from $A$ to $B$ preserving the respective Lip-norms \cite[Theorem 5.13]{Latremoliere16}.\\

By combining Remark~\ref{EstimateRemark} with estimates for maps induced by suitable positive definite functions on Coxeter groups, in the present section we investigate the dependence of the families of compact quantum metric spaces constructed in Section~\ref{sec:Iwahori--Hecke-Algebras} on the deformation parameter~$q$, with respect to the quantum Gromov--Hausdorff propinquity.

\vspace{3mm}


\subsection{Positive Definite Functions on Coxeter Groups}\label{subsec:Positive-Definite-Functions}

Let $(W, S)$ be a finite rank Coxeter system. As shown by Bo\.{z}ejko, Januszkiewicz, and Spatzier in \cite{BozejkoJanuszkiewiczSpatzier88}, the word length function associated with $S$ is \emph{conditionally negative definite}. By Schoenberg's theorem (see, e.g., \cite[Theorem D.11]{BrownOzawa08}), it follows that the map $W \ni \mathbf{w} \mapsto \kappa^{|\mathbf{w}|}$ is \emph{positive definite} for every $0 < \kappa \leq 1$, meaning that for any finite set of elements $\mathbf{w}_1, \ldots, \mathbf{w}_n \in W$, the matrix
\[
\left( \kappa^{|\mathbf{w}_i^{-1} \mathbf{w}_j|} \right)_{1 \leq i, j \leq n} \in M_n(\mathbb{C})
\]
is positive. Consequently, the associated \emph{Schur multiplier}
\[
m_{\kappa} : \mathcal{B}(\ell^2(W)) \to \mathcal{B}(\ell^2(W)), \quad
m_{\kappa}(x) := \left( \kappa^{|\mathbf{v}^{-1}\mathbf{w}|} x_{\mathbf{v}, \mathbf{w}} \right)_{\mathbf{v}, \mathbf{w} \in W}
\]
defines a unital, completely positive map on $\mathcal{B}(\ell^2(W))$. Here, we represent elements $x \in \mathcal{B}(\ell^2(W))$ as $W \times W$-matrices with entries given by $x_{\mathbf{v}, \mathbf{w}} := \langle x \delta_{\mathbf{v}}, \delta_{\mathbf{w}} \rangle$. For further background on positive definite functions and Schur multipliers, we refer to \cite{Davis08}.

We embed $\ell^{\infty}(W)$ into $\mathcal{B}(\ell^2(W))$ via pointwise multiplication, defining $f \delta_{\mathbf{w}} := f(\mathbf{w}) \delta_{\mathbf{w}}$ for all $f \in \ell^{\infty}(W)$ and $\mathbf{w} \in W$. With respect to this embedding, the Schur multiplier $m_{\kappa}$ is an \emph{$\ell^{\infty}(W)$-bimodule map}, i.e., $m_{\kappa}(f x g) = f m_{\kappa}(x) g$ for all $f, g \in \ell^{\infty}(W)$, $x \in \mathcal{B}(\ell^2(W))$. Moreover, for each $\mathbf{w} \in W$, we have $m_{\kappa}(T_{\mathbf{w}}^{(1)}) = \kappa^{|\mathbf{w}|} T_{\mathbf{w}}^{(1)}$, so that $m_{\kappa}$ restricts to a unital, completely positive map on the reduced group C$^*$-algebra $C_r^*(W) \subseteq \mathcal{B}(\ell^2(W))$.

Consider now the Dirac operator
\[
D_S := \sum_{n \in \mathbb{N}} n P_n
\]
introduced in Section~\ref{sec:Iwahori--Hecke-Algebras}, and let $C_c(W, \ell^{\infty}(W)) \subseteq \mathcal{B}(\ell^2(W))$ denote the $*$-algebra of finite sums of the form $\sum_{\mathbf{w} \in W} f_{\mathbf{w}} T_{\mathbf{w}}^{(1)}$ with $f_{\mathbf{w}} \in \ell^{\infty}(W)$. For every $0 < \kappa \leq 1$, the Schur multiplier $m_{\kappa}$ preserves this algebra:
\[
m_{\kappa}\left( C_c(W, \ell^{\infty}(W)) \right) \subseteq C_c(W, \ell^{\infty}(W)).
\]
Furthermore, since $T_s^{(q)} = T_s^{(1)} + p_s(q) Q_s$ for all $s \in S$ and $q \in \mathbb{R}_{>0}^{(W, S)}$, with $Q_s \in \ell^{\infty}(W)$, it follows that the Iwahori--Hecke algebra $\mathbb{C}_q[W]$ is contained in $C_c(W, \ell^{\infty}(W))$.

As stated in Section~\ref{sec:Iwahori--Hecke-Algebras}, the domain of $D_S$ is invariant under multiplication by elements of $\mathbb{C}_q[W]$, and the commutator $[D_S, x] := D_S x - x D_S$ is bounded for all $x \in \mathbb{C}_q[W]$. This property extends to all elements in $C_c(W, \ell^{\infty}(W))$, as demonstrated in the following lemma.

\begin{lemma}
Let $(W, S)$ be a finite rank Coxeter system, and let $x \in C_c(W, \ell^{\infty}(W))$. Then the commutator $[D_S, x]$ defines a bounded operator on $\ell^2(W)$, and
\[
\| [D_S, m_{\kappa}(x)] \| \leq \| [D_S, x] \|.
\]
\end{lemma}

\begin{proof}
Let $x = \sum_{\mathbf{w} \in W} f_{\mathbf{w}} T_{\mathbf{w}}^{(1)}$ be a finite sum with $f_{\mathbf{w}} \in \ell^{\infty}(W)$. Define, for each $\mathbf{w} \in W$, a function $\varphi_{\mathbf{w}} \in \ell^{\infty}(W)$ by $\varphi_{\mathbf{w}}(\mathbf{v}) := |\mathbf{v}| - |\mathbf{w}^{-1} \mathbf{v}|$, for all $\mathbf{v} \in W$. Then, for each $\mathbf{v} \in W$,
\[
[D_S, x] \delta_{\mathbf{v}} 
= \sum_{\mathbf{w} \in W} \left( |\mathbf{w} \mathbf{v}| - |\mathbf{v}| \right) f_{\mathbf{w}}(\mathbf{w} \mathbf{v}) \delta_{\mathbf{w} \mathbf{v}} 
= \sum_{\mathbf{w} \in W} \varphi_{\mathbf{w}}(\mathbf{w} \mathbf{v}) f_{\mathbf{w}}(\mathbf{w} \mathbf{v}) \delta_{\mathbf{w} \mathbf{v}}.
\]
Thus, the commutator can be expressed as
\[
[D_S, x] = \sum_{\mathbf{w} \in W} \varphi_{\mathbf{w}} f_{\mathbf{w}} T_{\mathbf{w}}^{(1)} \in C_c(W, \ell^{\infty}(W)),
\]
establishing boundedness.

For the second claim, we observe that $m_{\kappa}$ is a contraction and satisfies
\[
[D_S, m_{\kappa}(x)] = \sum_{\mathbf{w} \in W} \kappa^{|\mathbf{w}|} \varphi_{\mathbf{w}} f_{\mathbf{w}} T_{\mathbf{w}}^{(1)} = m_{\kappa}([D_S, x]),
\]
which completes the proof.
\end{proof}

We now establish an analogue of Theorem~\ref{HaagerupTypeCondition}, involving the family of unital completely positive maps $(m_{\kappa})_{0 < \kappa \leq 1}$. This result will enable us, in the next step, to derive appropriate estimates on the norms $\Vert m_{\kappa}(x - x^{(q')})\Vert$ and $\Vert [D_S, m_{\kappa}(x - x^{(q')})] \Vert$ for $q, q' \in \mathbb{R}_{>0}^{(W, S)}$ and $x \in \mathbb{C}_q[W]$, as will be seen in Proposition~\ref{MagnitudeEstimate}.

\begin{lemma} \label{HaagerupTypeCondition2}
Let $(W,S)$ be a finite rank, right-angled Coxeter system, and suppose that the graph $\Gamma$ defined in Theorem \ref{HaagerupTypeCondition} contains no induced square. Then there exists a constant $K > 0$ such that
\[
 \Vert P_{i} m_{\kappa}(x - x^{(q^{\prime})}) P_{j}  \Vert 
\leq \kappa^{|i-j|} K C_{q,q^{\prime}} \Vert x \delta_{e} \Vert_{2}
\]
for all $q, q^{\prime} \in \mathbb{R}_{>0}^{(W,S)}$, $i,j,n \in \mathbb{N}$, and $x \in \chi_{n}(\mathbb{C}_{q}[W])$, where
\[
C_{q,q^{\prime}} := \max_{\Gamma_{0} \in \emph{Cliq}(\Gamma)} \left| \prod_{t \in V\Gamma_{0}} p_{t}(q) - \prod_{t \in V\Gamma_{0}} p_{t}(q^{\prime}) \right|,
\]
and $x^{(q^{\prime})} := \sum_{\mathbf{w} \in W} x(\mathbf{w}) T_{\mathbf{w}}^{(q^{\prime})}$.
\end{lemma}

\begin{proof}
The statement follows similarly to the proof of Theorem \ref{HaagerupTypeCondition}. Let $\xi \in P_{j} \ell^{2}(W)$. By the preceding discussion, along with Proposition \ref{GeneratorDecomposition} and Remark \ref{OperatorRemark}, we have
\begin{eqnarray*}
& & P_{i}m_{\kappa}(x-x^{(q^{\prime})})\xi \\
&=& \sum_{(l,k,\Gamma_{0},\Gamma_{1},\Gamma_{2})}\sum_{\mathbf{w}\in W:|\mathbf{w}|=n}\kappa^{\left|\mathbf{w}_{l,K_{0},K_{1},K_{2}}^{(1)}\mathbf{w}_{l,K_{0},K_{1},K_{2}}^{(3)}\right|}\left(\prod_{t\in V\Gamma_{0}}p_{t}(q)-\prod_{t\in V\Gamma_{0}}p_{t}(q^{\prime})\right) \\
& & \qquad \times \left(\sum\left\{ \left.x(\mathbf{w})\xi(\mathbf{v})\delta_{\mathbf{w}_{l,K_{0},K_{1},K_{2}}^{(1)}\mathbf{w}_{l,K_{0},K_{1},K_{2}}^{(3)}\mathbf{v}}\,\right| \mathbf{v}\in W,\left|\mathbf{v}\right|=j, (\mathbf{w}_{l,K_{0},K_{1},K_{2}}^{(3)})^{-1}\leq\mathbf{v} \right.\right.\\
& & \qquad  \qquad  \qquad \qquad  \left.\left.\mathbf{w}_{l,K_{0},K_{1},K_{2}}^{(2)}\leq\mathbf{w}_{l,K_{0},K_{1},K_{2}}^{(3)}\mathbf{v},\mathbf{w}_{l,K_{0},K_{1},K_{2}}^{(1)}\leq\mathbf{w}_{l,K_{0},K_{1},K_{2}}^{(1)}\mathbf{w}_{l,K_{0},K_{1},K_{2}}^{(3)}\mathbf{v}\right\} \right),
\end{eqnarray*}
where the outer sum runs over all tuples $(l,\Gamma_{0},\Gamma_{1},\Gamma_{2})$ with $0\leq l\leq n$, $\Gamma_{0}\in\mathrm{Cliq}(\Gamma,l)$, and $(\Gamma_{1},\Gamma_{2})\in\mathrm{Comm}(\Gamma_{0})$, and where the elements $\mathbf{w}_{l,\Gamma_{0},\Gamma_{1},\Gamma_{2}}^{(1)}$, $\mathbf{w}_{l,\Gamma_{0},\Gamma_{1},\Gamma_{2}}^{(2)}$ and $\mathbf{w}_{l,\Gamma_{0},\Gamma_{1},\Gamma_{2}}^{(3)}$ are defined as in the proof of Theorem \ref{HaagerupTypeCondition}.

As before, we obtain
\begin{equation}
\begin{split} & \left|(P_{i}a\xi)(\mathbf{v})\right|\\
 & \leq(\#\text{Cliq}(\Gamma))^{2}C_{q,q^{\prime}}\sum_{\Gamma_{0}\in\text{Cliq}(\Gamma)}\left(\sum\left\{ \left.\kappa^{|\mathbf{u}\mathbf{u}^{\prime}|}|x(\mathbf{u}V(\Gamma_{0})\mathbf{u}^{\prime})||\xi((\mathbf{u}\mathbf{u}^{\prime})^{-1}\mathbf{v})|\:\right|\right.\,\right.\mathbf{u},\mathbf{u}^{\prime}\in W\text{ with }\\
 & \quad \qquad\qquad\qquad\qquad\qquad\qquad\left|\mathbf{u}\right|=2^{-1}(i-j+n-\#\Gamma_{0}),\left|\mathbf{u}^{\prime}\right|=n-2^{-1}(i-j+n+\#\Gamma_{0}),\\
 & \qquad\qquad\qquad\qquad\qquad\qquad\qquad\qquad\qquad\qquad\qquad\left.\left.(\mathbf{u}^{\prime})^{-1}\leq(\mathbf{u}\mathbf{u}^{\prime})^{-1}\mathbf{v},V(\Gamma_{0})\leq\mathbf{u}^{-1}\mathbf{v},\mathbf{u}\leq\mathbf{v}\right\} \right).
\end{split}
\label{eq:TupleDefinition-1}
\end{equation}
For such $\mathbf{u}, \mathbf{u}'$, note that $|\mathbf{u} \mathbf{u}'| \geq |\,|\mathbf{u}| - |\mathbf{u}'|\,| = |i-j|$. Applying the Cauchy–Schwarz inequality yields
\begin{eqnarray*}
\|P_{i}x\xi\|_{2}^{2} &\leq& \kappa^{2|i-j|}(\#\mathrm{Cliq}(\Gamma))^{4}C_{q,q^{\prime}}^{2}\sum_{\mathbf{v}\in W:|\mathbf{v}|=i}\left(\sum_{(\Gamma_{0},\mathbf{u},\mathbf{u}')}|x(\mathbf{u}V(\Gamma_{0})\mathbf{u}')|^{2}\right)\left(\sum_{(\Gamma_{0},\mathbf{u},\mathbf{u}')}|\xi((\mathbf{u}\mathbf{u}')^{-1}\mathbf{v})|^{2}\right) \\
&\leq& \kappa^{2|i-j|}(\#\mathrm{Cliq}(\Gamma))^{4}C_{q,q^{\prime}}^{2}\sum_{\mathbf{x}\in W:|\mathbf{x}|=n}\sum_{\mathbf{y}\in W:|\mathbf{y}|=j}R_{\mathbf{x},\mathbf{y}}\,|x(\mathbf{x})|^{2}\,|\xi(\mathbf{y})|^{2},
\end{eqnarray*}
where the sums in the first line run over tuples as in the sums in \eqref{eq:TupleDefinition-1} and where $R_{\mathbf{x},\mathbf{y}}$ is as in Lemma \ref{SetSize}. Hence, using the constant $K$ from Lemma \ref{SetSize},
\[
\|P_{i} x \xi\|_{2} \leq \kappa^{|i-j|} \sqrt{K} (\#\text{Cliq}(\Gamma))^{2} C_{q,q^{\prime}} \|x \delta_{e}\|_{2} \|\xi\|_{2},
\]
as required.
\end{proof}

\begin{proposition} \label{MagnitudeEstimate}
Let $(W,S)$ be a finite rank, right-angled Coxeter system, and suppose the graph $\Gamma$ from Theorem \ref{HaagerupTypeCondition} contains no induced square. Then there exists a constant $K > 0$ such that for all $q, q^{\prime} \in \mathbb{R}_{>0}^{(W,S)}$, $0 < \kappa \leq 1$, and $x \in \mathbb{C}_{q}[W]$, the following estimates hold:
\[
\Vert m_{\kappa}(x - x^{(q^{\prime})}) \Vert \leq \frac{K}{1 - \kappa} C_{q,q^{\prime}} L_{S}(x), \quad
\Vert [D_{S}, m_{\kappa}(x - x^{(q^{\prime})})] \Vert \leq \frac{\kappa K}{(1 - \kappa)^2} C_{q,q^{\prime}} L_{S}(x),
\]
where $C_{q,q^{\prime}}$ is as in Lemma \ref{HaagerupTypeCondition2}, and $x^{(q^{\prime})} := \sum_{\mathbf{w} \in W} x(\mathbf{w}) T_{\mathbf{w}}^{(q^{\prime})}$.
\end{proposition}

\begin{proof}
Let $K$ be the constant from Lemma \ref{HaagerupTypeCondition2}. For any $x \in \mathbb{C}_{q}[W]$, write
\[
m_{\kappa}(x - x^{(q^{\prime})}) = \sum_{|i| \leq N} \sum_{j \geq \max\{0, i\}} P_{j} m_{\kappa}(x - x^{(q^{\prime})}) P_{j-i}
\]
for sufficiently large $N \in \mathbb{N}$, with the sum converging in the strong operator topology. Lemma \ref{HaagerupTypeCondition2} implies
\begin{align*}
\Vert m_{\kappa}(x - x^{(q^{\prime})}) \Vert 
&\leq \sum_{|i| \leq N} \sup_{j \geq \max\{0, i\}} \Vert P_{j} m_{\kappa}(x - x^{(q^{\prime})}) P_{j-i} \Vert \\
&\leq KC_{q,q^{\prime}} \left( \sum_{|i| \leq N} \kappa^{|i|} \right) \left( \sum_{n = 0}^{\infty} \Vert \chi_{n}(x) \delta_{e} \Vert_{2} \right).
\end{align*}
Using the bound $\sum_{|i| \leq N} \kappa^{|i|} \leq 2/(1 - \kappa)$, the Cauchy–Schwarz inequality, and \eqref{eq:2-Norm}, we find
\[
\Vert m_{\kappa}(x - x^{(q^{\prime})}) \Vert 
\leq \frac{2\pi}{\sqrt{6}(1 - \kappa)} K C_{q,q^{\prime}} \Vert [D_{S}, x] \delta_{e} \Vert_{2}
\leq \frac{2\pi}{\sqrt{6}(1 - \kappa)} K C_{q,q^{\prime}} L_{S}(x).
\]

For the second estimate, observe that
\[
[D_{S}, m_{\kappa}(x - x^{(q^{\prime})})] 
= \sum_{|i| \leq N} \sum_{j \geq \max\{0, i\}} i P_{j} m_{\kappa}(x - x^{(q^{\prime})}) P_{j-i}.
\]
Arguing as above,
\begin{align*}
\Vert [D_{S}, m_{\kappa}(x - x^{(q^{\prime})})] \Vert 
&\leq KC_{q,q^{\prime}} \left( \sum_{|i| \leq N} |i| \kappa^{|i|} \right) \left( \sum_{n = 0}^{\infty} \Vert \chi_{n}(x) \delta_{e} \Vert_{2} \right) \\
&\leq \frac{2\pi \kappa}{\sqrt{6}(1 - \kappa)^2} K C_{q,q^{\prime}} \Vert [D_{S}, x] \delta_{e} \Vert_{2} \\
&\leq  \frac{2\pi \kappa}{\sqrt{6}(1 - \kappa)^2} K C_{q,q^{\prime}} L_{S}(x),
\end{align*}
which completes the proof.
\end{proof}

\vspace{3mm}


\subsection{Continuity at \texorpdfstring{$q=1$}{q=1}}

Let $(W, S)$ be a finite rank, right-angled Coxeter system, and assume that the graph $\Gamma$ appearing in Theorem~\ref{HaagerupTypeCondition} contains no induced squares. Then, by Theorem~\ref{CQMSStatement}, for every deformation parameter $q \in \mathbb{R}_{>0}^{(W,S)}$, the associated pair $(C_{r,q}^{\ast}(W), L_{S}^{(q)})$ forms a compact quantum metric space. It is straightforward to verify that $L_{S}^{(q)}$ is lower semi-continuous and satisfies the Leibniz property, implying that the pair defines a Leibniz quantum compact metric space.

In particular, one may consider the quantum Gromov--Hausdorff propinquity between these spaces, which provides a natural framework for analyzing the continuity of the deformation with respect to the parameter $q$.

\begin{question}
Let $(W, S)$ be a finite rank, right-angled Coxeter system, and suppose that the graph $\Gamma$ defined in Theorem~\ref{HaagerupTypeCondition} contains no induced square. Is the map
\[
\mathbb{R}_{>0}^{(W, S)} \ni q \longmapsto (C_{r,q}^{\ast}(W), L_{S}^{(q)})
\]
continuous with respect to the quantum Gromov--Hausdorff propinquity?
\end{question}

The results established in the preceding subsection allow us to affirm the continuity of this map at the point $q = 1$.

\begin{theorem} \label{ConvergenceTheorem}
Let $(W, S)$ be a finite rank, right-angled Coxeter system, and assume that the graph $\Gamma$ defined in Theorem~\ref{HaagerupTypeCondition} contains no induced square. Then,
\[
(C_{r,q}^{\ast}(W), L_{S}^{(q)}) \longrightarrow (C_{r}^{\ast}(W), L_{S}^{(1)})
\]
in the quantum Gromov--Hausdorff propinquity as $q \to 1$.
\end{theorem}

The proof of Theorem~\ref{ConvergenceTheorem} hinges on Remark~\ref{EstimateRemark}, which states that the length of any bridge provides an upper bound for the quantum propinquity. Thus, the primary task is to construct, for each $q$ sufficiently close to $1$, a bridge between $C_{r,q}^{\ast}(W)$ and $C_{r}^{\ast}(W)$ whose reach and height both tend to zero in the limit.

To this end, consider the unital C$^*$-algebra $\mathcal{B}(\ell^2(W))$ and the canonical inclusions
\[
\pi_q : C_{r,q}^{\ast}(W) \hookrightarrow \mathcal{B}(\ell^2(W)), \quad
\pi_1 : C_{r,1}^{\ast}(W) \hookrightarrow \mathcal{B}(\ell^2(W)).
\]
Then the quadruple $\gamma := (\mathcal{B}(\ell^2(W)), 1, \pi_q, \pi_1)$ defines a bridge from $C_{r,q}^{\ast}(W)$ to $C_{r}^{\ast}(W)$. Observe that the height of $\gamma$ with respect to the Lip-norms $L_{S}^{(q)}$ and $L_{S}^{(1)}$ is zero. Consequently,
\[
\lambda(\gamma \mid L_{S}^{(q)}, L_{S}^{(1)}) 
= \mathfrak{reach}(\gamma \mid L_{S}^{(q)}, L_{S}^{(1)}) 
= \max \left\{ \sup_{x \in \mathcal{L}_1^{(1)}} \mathrm{dist}(x, \mathcal{L}_1^{(q)}), 
\sup_{x \in \mathcal{L}_1^{(q)}} \mathrm{dist}(x, \mathcal{L}_1^{(1)}) \right\},
\]
where
\[
\mathcal{L}_1^{(q)} := \{ x \in \mathbb{C}_q[W] \mid L_{S}^{(q)}(x) \leq 1 \}, \qquad 
\mathcal{L}_1^{(1)} := \{ x \in \mathbb{C}[W] \mid L_{S}^{(1)}(x) \leq 1 \}.
\]

It therefore suffices to show that both suprema on the right-hand side converge to zero as $q \to 1$. This is the content of the following two propositions.

\begin{proposition}
Let $(W,S)$ be a finite rank, right-angled Coxeter system, and suppose that the graph $\Gamma$ defined in Theorem \ref{HaagerupTypeCondition} contains no induced square. Then
\[
\sup_{x \in \mathcal{L}_1^{(1)}} \emph{dist}(x, \mathcal{L}_1^{(q)}) \longrightarrow 0 \quad \text{as } q \to 1.
\]
\end{proposition}

\begin{proof}
Let $\varepsilon > 0$. By Proposition \ref{UnityBoundedness}, the set $\mathcal{B}_1^{(1)} := \mathcal{L}_1(C_r^*(W), L_S, \tau_1)$ is totally bounded. Hence, it can be covered by finitely many $\varepsilon/2$-balls centered at elements $x_1, \dots, x_n \in \mathcal{B}_1^{(1)}$. For each $1 \leq i \leq n$ and $q \in \mathbb{R}_{>0}^{(W,S)}$, define
\[
x_i^{(q)} := \sum_{\mathbf{w} \in W} x_i(\mathbf{w}) T_{\mathbf{w}}^{(q)} \in \mathbb{C}_q[W].
\]
Since $x_i \notin \mathbb{C}1$, we may define
\[
y_i^{(q)} := \frac{L_S^{(1)}(x_i)}{L_S^{(q)}(x_i^{(q)})} \, x_i^{(q)} \in \mathcal{B}_1^{(q)}.
\]
It follows from Proposition \ref{GeneratorDecomposition}, that the map $q \mapsto y_i^{(q)}$ is continuous and satisfies $\|x_i - y_i^{(q)}\| \to 0$ as $q \to 1$. Thus, there exists an open neighborhood $\mathcal{U} \subseteq \mathbb{R}_{>0}^{(W,S)}$ of $1$ such that $\|x_i - y_i^{(q)}\| < \varepsilon/2$ for all $1 \leq i \leq n$ and $q \in \mathcal{U}$.

Now let $x \in \mathcal{L}_1^{(1)}$ be arbitrary. Then $x - \tau_1(x) \in \mathcal{B}_1^{(1)}$, so for some $1 \leq i \leq n$ we have $\|x - \tau_1(x) - x_i\| < \varepsilon/2$. For $q \in \mathcal{U}$, we estimate
\[
\|x - (y_i^{(q)} + \tau_1(x))\| \leq \|x - \tau_1(x) - x_i\| + \|x_i - y_i^{(q)}\| < \varepsilon.
\]
Since $y_i^{(q)} + \tau_1(x) \in \mathcal{L}_1^{(q)}$ and $x \in \mathcal{L}_1^{(1)}$ was arbitrary, it follows that $\sup_{x \in \mathcal{L}_1^{(1)}} \text{dist}(x, \mathcal{L}_1^{(q)}) < \varepsilon$ for all $q \in \mathcal{U}$. The claim follows.
\end{proof}

\begin{proposition}
Let $(W,S)$ be a finite rank, right-angled Coxeter system, and suppose that the graph $\Gamma$ defined in Theorem \ref{HaagerupTypeCondition} contains no induced square. Then
\[
\sup_{x \in \mathcal{L}_1^{(q)}} \emph{dist}(x, \mathcal{L}_1^{(1)}) \longrightarrow 0 \quad \text{as } q \to 1.
\]
\end{proposition}

\begin{proof}
Let $(q_i)_{i \in \mathbb{N}} \subseteq \mathbb{R}_{>0}^{(W,S)}$ be a sequence with $q_i \to 1$, and let $\mathcal{K} \subseteq \mathbb{R}_{>0}^{(W,S)}$ be a compact neighborhood of $1$ containing all $q_i$. Fix $\varepsilon > 0$. By Proposition \ref{UnityBoundedness} and the structure of the Schur multipliers from Subsection \ref{subsec:Positive-Definite-Functions}, we may choose $0 < \kappa < 1$ such that $\sup_{x \in \mathbf{B}} \|x - m_\kappa(x)\| < \frac{\varepsilon}{2}$, where $\mathbf{B} := \bigcup_{q \in \mathcal{K}} \mathcal{B}_1^{(q)}$. Since $\mathbf{B}$ is totally bounded, there exists $R > 0$ such that $\|x\| \leq R$ for all $x \in \mathbf{B}$.

Let $K > 0$ be the constant from Proposition \ref{MagnitudeEstimate}, define
\[
C_{q,q'} := \max_{\Gamma_0 \in \text{Cliq}(\Gamma)} \left| \prod_{t \in V\Gamma_0} p_t(q) - \prod_{t \in V\Gamma_0} p_t(q') \right|,
\]
and set
\[
F(q,q') := \frac{\kappa K}{(1 - \kappa)^2} C_{q,q'}.
\]
Choose $i_0 \in \mathbb{N}$ such that for all $i \geq i_0$,
\[
\frac{F(q_i,1)}{1 + F(q_i,1)} \left( R + \frac{1 - \kappa}{\kappa} \right) < \frac{\varepsilon}{2}.
\]
We claim that $\sup_{x \in \mathcal{L}_1^{(q_i)}} \text{dist}(x, \mathcal{L}_1^{(1)}) < \varepsilon$ for all $i \geq i_0$. Indeed, let $x \in \mathcal{L}_1^{(q_i)}$ be arbitrary and set $\overline{x} := x - \tau_{q_i}(x) \in \mathcal{B}_1^{(q_i)}$. Then $\|\overline{x} - m_\kappa(\overline{x})\| < \varepsilon/2$ and $\|\overline{x}\| \leq R$.

For $q \in \mathbb{R}_{>0}^{(W,S)}$, define $\overline{x}^{(q)} := \sum_{\mathbf{w} \in W \setminus\{e\}} x(\mathbf{w}) T_{\mathbf{w}}^{(q)} \in \mathbb{C}_q[W]$, and set
\[
y^{(q)} := \tau_{q_i}(x) + \left(1 + F(q_i, q)\right)^{-1} m_\kappa(\overline{x}^{(q)}) \in \mathcal{B}(\ell^2(W)).
\]
By Proposition \ref{MagnitudeEstimate} and the reverse triangle inequality,
\[
| \|[D_S, m_\kappa(\overline{x})]\| - \|[D_S, m_\kappa(\overline{x}^{(1)})]\| | \leq \|[D_S, m_\kappa(\overline{x} - \overline{x}^{(1)})]\| \leq F(q_i,1),
\]
so that
\[
\|[D_S, m_\kappa(\overline{x}^{(1)})]\| \leq 1 + F(q_i,1),
\]
implying $y^{(1)} \in \mathcal{L}_1^{(1)}$. Moreover, $y^{(q_i)} = m_\kappa(x)$. Using Proposition \ref{MagnitudeEstimate} again, we estimate:
\begin{align*}
\|x - y^{(1)}\| &\leq \|\overline{x} - m_\kappa(\overline{x})\| + \left\| m_\kappa(\overline{x}) - \left(1 + F(q_i,1)\right)^{-1} m_\kappa(\overline{x}^{(1)}) \right\| \\
&\leq \frac{\varepsilon}{2} + \frac{1}{1 + F(q_i,1)} \left( F(q_i,1) \|m_\kappa(\overline{x})\| + \|m_\kappa(\overline{x} - \overline{x}^{(1)})\| \right) \\
&\leq \frac{\varepsilon}{2} + \frac{F(q_i,1)}{1 + F(q_i,1)} \left( \|m_\kappa(\overline{x})\| + \frac{1 - \kappa}{\kappa} \right) \\
&\leq \frac{\varepsilon}{2} + \frac{F(q_i,1)}{1 + F(q_i,1)} \left( R + \frac{1 - \kappa}{\kappa} \right) \\
&<  \varepsilon.
\end{align*}
This shows that $\sup_{x \in \mathcal{L}_1^{(q_i)}} \text{dist}(x, \mathcal{L}_1^{(1)}) < \varepsilon$ for all $i \geq i_0$. Since $\varepsilon > 0$ was arbitrary, the result follows.
\end{proof}

\vspace{3mm}


\section{Final Remarks} \label{Section4}

\vspace{3mm}

Several natural questions arise from the results of this article, pointing toward various directions for future research -- whether through modifications of our initial assumptions, extensions to more general settings, or explorations of novel configurations. The aim of this concluding section is to reflect on our main results and methods, to acknowledge their limitations, and to outline potential avenues for further investigation.


\subsection{The Haagerup-type Condition} Building on \cite{OzawaRieffel05}, this work focuses on Iwahori--Hecke algebras associated with finite rank, right-angled Coxeter systems whose associated graphs contain no induced square. In this setting, the decomposition of generators given in Proposition~\ref{GeneratorDecomposition}, combined with the combinatorial analysis of Lemma~\ref{SetSize}, enabled us to verify that the canonical $\ast$-filtrations on these algebras satisfy the Haagerup-type condition. By the main result of \cite{OzawaRieffel05}, this in turn gives rise to compact quantum metric space structures. To the best of the authors' knowledge, decompositions of generators analogous to Proposition~\ref{GeneratorDecomposition} have not yet been explored outside the right-angled case. It would be of interest to investigate whether such decompositions exist for other classes of Coxeter groups and whether they can be fruitfully applied in the general setting of Section~\ref{sec:Iwahori--Hecke-Algebras}.

\begin{question} \label{HaagerupQuestion}
Let $(W, S)$ be a Coxeter system and $q \in \mathbb{R}_{>0}^{(W,S)}$. Under what conditions does the canonical $\ast$-filtration of the Iwahori--Hecke algebra $\mathbb{C}_{q}[W]$, as constructed in Section~\ref{sec:Iwahori--Hecke-Algebras}, satisfy the Haagerup-type condition -- that is, when does an analogue of Theorem~\ref{HaagerupTypeCondition}\eqref{FirstStatement} hold?
\end{question}

Based on the characterizations provided in Theorem~\ref{HaagerupTypeCondition} and Theorem~\ref{MoussongTheorem}, we expect an affirmative answer to Question~\ref{HaagerupQuestion} if and only if the Coxeter group $W$ is word-hyperbolic. 

In \cite[Section~6]{OzawaRieffel05}, it was shown that the reduced free product of filtered C$^{\ast}$-algebras satisfies the Haagerup-type condition whenever its constituent algebras do. In \cite[Question~6.10]{OzawaRieffel05}, the authors ask whether this result extends to amalgamated free products of C$^{\ast}$-algebras. Notably, every Hecke C$^{\ast}$-algebra associated with a right-angled Coxeter system admits a decomposition as an amalgamated free product of Hecke C$^{\ast}$-algebras arising from certain Coxeter subsystems (see \cite[Theorem~2.15]{CaspersFima17} and the proof of \cite[Corollary~3.4]{Caspers20}). Our Theorem~\ref{HaagerupTypeCondition} hence shows that the Haagerup-type condition is not preserved in general under such amalgamations. However, we expect that preservation does hold when amalgamating over finite-dimensional subalgebras. Utilizing the C$^{\ast}$-algebraic analogue of the decomposition in \cite[Proposition~8.3]{CaspersKlisseSkalskiVosWasilewski23}, together with \cite[Proposition~8.8.5]{Davis08}, one may deduce that the canonical $\ast$-filtrations on Iwahori--Hecke algebras of virtually free Coxeter groups satisfy the Haagerup-type condition.

Another generalization of free products is given by reduced graph products of C$^{\ast}$-algebras, as introduced in \cite{CaspersFima17}. This construction associates to a simplicial graph and a collection of C$^{\ast}$-algebras (each endowed with a GNS-faithful state) at each vertex a new C$^{\ast}$-algebra that blends its components according to the graph’s structure. It interpolates between Voiculescu's reduced free product \cite{Voiculescu85} (see also \cite{Avitzour82}) and the minimal tensor product, while preserving essential properties. In particular, every right-angled Hecke C$^{\ast}$-algebra decomposes as a graph product of 2-dimensional C$^{\ast}$-algebras over the graph $\Gamma$ defined in Theorem~\ref{HaagerupTypeCondition} (see again \cite[Corollary~3.4]{Caspers20}). In analogy with the construction in \cite[Section~6]{OzawaRieffel05}, if the vertex algebras admit $\ast$-filtrations, then one may construct a $\ast$-filtration on the entire graph product. This naturally leads to the question of whether the Haagerup-type condition is preserved under such graph product constructions. By \cite[Proposition~2.6]{CaspersKlisseLarsen21}, reduced elements in graph products admit decompositions similar to that in Proposition~\ref{GeneratorDecomposition}. Using this in conjunction with suitable combinatorial estimates, we expect a generalization of Theorem~\ref{HaagerupTypeCondition} to hold.


\subsection{Compact Quantum Metric Structures}

The main motivation for studying the Haagerup-type condition lies in its connection with compact quantum metric spaces. As spelled out in Section \ref{sec:Iwahori--Hecke-Algebras}, we therefore extend Question \ref{HaagerupQuestion} to the one for Lip-norms on Hecke C$^{\ast}$-algebras.

\begin{question} \label{MainQuestion2}
Let $(W, S)$ be a Coxeter system and $q \in \mathbb{R}_{>0}^{(W,S)}$. Under what conditions does the pair $(C_{r,q}^{*}(W), L_{S}^{(q)})$ defined in Section~\ref{sec:Iwahori--Hecke-Algebras} define a compact quantum metric space?
\end{question}

Recall that when $q = 1$, the Hecke C$^{\ast}$-algebra $C_{r,1}^{*}(W)$ reduces to the group C$^{\ast}$-algebra $C_r^*(W)$ of the Coxeter group $W$. Lipschitz seminorms arising from Dirac operators associated with word-length functions on reduced group C$^{\ast}$-algebras of finitely generated groups have been studied in \cite{Rieffel02, OzawaRieffel05, ChristRieffel}. To the authors' knowledge, the only known results on Lip-norms in this context pertain to groups of polynomial growth (see \cite{Rieffel02, ChristRieffel}) and to word-hyperbolic groups (see \cite{OzawaRieffel05}). Besides the immediate question for an extension of these results to larger classes of groups, it would be interesting to see whether the ideas in \cite{Rieffel02} and \cite{ChristRieffel} may be applied in the context of Iwahori--Hecke algebras. One natural class to consider is that of \emph{affine type} Coxeter systems, for which the underlying Coxeter group is infinite and virtually Abelian. Such groups admit a concrete geometric realization as reflection groups acting on Euclidean spaces (see \cite[Chapters~1 and 4]{Humphreys90}). The associated Iwahori--Hecke algebras have a rich structure and admit a so-called \emph{Bernstein presentation} in terms of crystallographic root systems (see \cite{Lusztig89, Bump10, Haines10}). In particular, they are finitely generated over their centers, a feature that may prove useful in addressing Question~\ref{MainQuestion2}.


\subsection{Continuity in the Quantum Gromov--Hausdorff Propinquity}

In addition to Theorem~\ref{HaagerupTypeCondition}, a central result of this paper is the continuity in Theorem~\ref{ConvergenceTheorem} for finite rank, right-angled, word-hyperbolic Coxeter systems. The proof relies on a family $(m_\kappa)_{0 < \kappa \leq 1}$ of Schur multipliers induced by positive definite functions on Coxeter groups, along with a variant of the inequality in Theorem~\ref{HaagerupTypeCondition}; see Subsection~\ref{subsec:Positive-Definite-Functions}. A key ingredient is the fact that $m_\kappa$ maps the reduced group C$^{\ast}$-algebra into itself for each $\kappa $. Unfortunately, this property does not extend to the $q$-deformed Hecke C$^{\ast}$-algebras, which obstructs a direct extension of the proof to the case $q \neq 1$. A possible approach is to suitably modify the family $(m_\kappa)_{0<\kappa \leq 1}$ to circumvent this issue, leading to the following question.

\begin{question}
Let $(W, S)$ be a finite rank, right-angled Coxeter system, and suppose the graph $\Gamma$ defined in Theorem~\ref{HaagerupTypeCondition} contains no induced square. Is the map
\[
\mathbb{R}_{>0}^{(W, S)} \ni q \mapsto (C_{r,q}^{\ast}(W), L_{S}^{(q)})
\]
continuous with respect to the quantum Gromov--Hausdorff propinquity?
\end{question}

We expect that, using similar techniques to the ones in Section \ref{qGHContinuity}, an analogue of Theorem \ref{ConvergenceTheorem} can be obtained for virtually free Coxeter groups.

\vspace{3mm}



\begin{thebibliography}{10}

\bibitem{AguilarHartglassPenneys22} K. Aguilar, M. Hartglass, D. Penneys, \emph{Compact quantum metric spaces from free graph algebras}, Internat. J. Math. 33 (2022), no. 10-11, Paper No. 2250073, 23 pp.

\bibitem{AguilarKaadKyed22} K. Aguilar, J. Kaad, D. Kyed, \emph{The Podleś spheres converge to the sphere}, Comm. Math. Phys. 392 (2022), no. 3, 1029–1061.

\bibitem{AustadKyed25} A. Austad, A. Kyed, \emph{Quantum metrics from length functions on quantum groups}, arXiv preprint arXiv:2503.01501 (2025).

\bibitem{Avitzour82}  D. Avitzour, \emph{Free products of C$^{\ast}$-algebras}, Trans. Amer. Math. Soc. 271 (1982), no. 2, 423--435.

\bibitem{BaumConnesHigson94} P. Baum, A. Connes, N. Higson, \emph{Classifying space for proper actions and K -
theory of group C$^\ast$-algebras}, C$^\ast$-algebras: 1943-1993 (San Antonio, TX, 1993), 240–291, Contemp. Math., 167, Amer. Math. Soc., Providence, RI, 1994.

\bibitem{Bernstein84} J. Bernstein, \emph{Le “centre” de Bernstein, Representations of reductive groups over a local field}, 1–32, Travaux en Cours, Hermann, Paris, 1984.

\bibitem{BjoernerBrenti05} A. Björner, F. Brenti, \emph{Combinatorics of Coxeter groups}, Graduate Texts in Mathematics, 231. Springer, New York, 2005. xiv+363 pp.

\bibitem{BorstCaspersWasilewski24} M. Borst, M. Caspers, M. Wasilewski, \emph{Bimodule coefficients, Riesz transforms on Coxeter groups and strong solidity}, Groups Geom. Dyn. 18 (2024), no. 2, 501–549.

\bibitem{Bourbaki02} N. Bourbaki, \emph{Lie groups and Lie algebras. Chapters 4-6}, Translated from the 1968 French original by Andrew Pressley, Elements of Mathematics (Berlin), Springer-Verlag, Berlin, 2002. xii+300 pp.

\bibitem{BozejkoJanuszkiewiczSpatzier88} M. Bo\.{z}ejko, T. Januszkiewicz, R. Spatzier, \emph{Infinite Coxeter groups do not have Kazhdan's property}, J. Operator Theory 19 (1988), no. 1, 63--67.

\bibitem{BrownOzawa08} N. Brown, N. Ozawa, \emph{C$^{\ast}$-algebras and finite-dimensional approximations}, Graduate Studies in Mathematics, 88. American Mathematical Society, Providence, RI, 2008. xvi+509 pp.

\bibitem{Bump10} D. Bump, Hecke algebras, http://sporadic.stanford.edu/bump/ math263/hecke.pdf (2010).

\bibitem{Caspers20} M. Caspers, \emph{Absence of Cartan subalgebras for right-angled Hecke von Neumann algebras}, Anal. PDE 13 (2020), no. 1, 1--28.

\bibitem{CaspersFima17} M. Caspers, P. Fima, \emph{Graph products of operator algebras}, J. Noncommut. Geom.11(2017), no.1, 367--411.

\bibitem{CaspersKlisseLarsen21} M. Caspers, M. Klisse, N.S. Larsen, \emph{Graph product Khintchine inequalities and Hecke C$^{\ast}$-algebras: Haagerup inequalities, (non)simplicity, nuclearity and exactness}, J. Funct. Anal. 280 (2021), no. 1, Paper No. 108795, 41 pp.

\bibitem{CaspersKlisseSkalskiVosWasilewski23} M. Caspers, M. Klisse, A. Skalski, G. Vos, M. Wasilewski, \emph{Relative Haagerup property for arbitrary von Neumann algebras}, Adv. Math. 421 (2023), Paper No. 109017, 61 pp.

\bibitem{CaspersSkalskiWasilewski19} M. Caspers, A. Skalski, M. Wasilewski, \emph{On masas in $q$-deformed von Neumann algebras}, Pacific J. Math. 302 (2019), no. 1, 1–21.

\bibitem{ChristRieffel}  M. Christ, M. A. Rieffel, \emph{Nilpotent group C$^{\ast}$-algebras as compact quantum metric spaces}, Canad. Math. Bull. 60 (2017), no. 1, 77-94.

\bibitem{Connes89} A. Connes, \emph{Compact metric spaces, Fredholm modules, and hyperfiniteness,} Ergodic Theory Dynam. Systems 9 (1989), no. 2, 207--220.

\bibitem{Connes94} A. Connes, \emph{Noncommutative geometry}, Academic Press, Inc., San Diego, CA, 1994. xiv+661 pp.

\bibitem{Davis08}  M. W. Davis, \emph{The Geometry and Topology of Coxeter groups}, London Mathematical Society Monographs Series, 32. Princeton University Press, Princeton, NJ, 2008. xvi+584pp.

\bibitem{DavisDymaraJanuszkiewiczOkun07} M. W. Davis, J. Dymara, T. Januszkiewicz, B. Okun, \emph{Weighted $L^2$-cohomology of Coxeter groups}, Geom. Topol. 11 (2007), 47–138.

\bibitem{DelormeOpdam08} P. Delorme, E. M. Opdam, \emph{The Schwartz algebra of an affine Hecke algebra}, J.
Reine Angew. Math. 625 (2008), 59–114.

\bibitem{Dykema93} K. J. Dykema, \emph{Free products of hyperfinite von Neumann algebras and free dimension}, Duke Math. J. 69 (1993), no. 1, 97–119

\bibitem{Dymara06} J. Dymara, \emph{Thin buildings}, Geom. Topol. 10 (2006), 667--694.

\bibitem{Garncarek16} L. Garncarek, \emph{Factoriality of Hecke-von Neumann algebras of right-angled Coxeter groups}, J. Funct. Anal. 270 (2016), no. 3, 1202--1219.

\bibitem{Gromov87} M. Gromov, \emph{Hyperbolic groups}, Essays in group theory, 75–263, Math. Sci. Res. Inst. Publ., 8, Springer, New York, 1987.

\bibitem{Haagerup78} U. Haagerup, \emph{An example of a nonnuclear C$^{\ast}$-algebra, which has the metric approximation property}, Invent. Math. 50 (1978/79), no. 3, 279--293.

\bibitem{Haines10} T. J. Haines, R. E. Kottwitz, A. Prasad, \emph{Iwahori--Hecke algebras}, J. Ramanujan Math. Soc. 25 (2010), no. 2, 113--145.

\bibitem{HivertThiery09} F. Hivert, N. Thiéry, \emph{The Hecke group algebra of a Coxeter group and its representation theory}, J. Algebra 321 (2009), no. 8, 2230–2258.

\bibitem{HivertSchillingThiery13} F. Hivert, A. Schilling, N. Thiéry, \emph{The biHecke monoid of a finite Coxeter group
and its representations}, Algebra Number Theory 7 (2013), no. 3, 595–671.

\bibitem{Humphreys90} J. Humphreys, \emph{Reflection groups and Coxeter groups}, Cambridge Stud. Adv. Math., 29 Cambridge University Press, Cambridge, 1990. xii+204 pp.

\bibitem{Iwahori64} N. Iwahori, \emph{On the structure of a Hecke ring of a Chevalley group over a finite field}, J. Fac. Sci. Univ. Tokyo Sect. I 10 (1964), 215–236.

\bibitem{IwahoriMatsumoto65} N. Iwahori, H. Matsumoto, \emph{On some Bruhat decomposition and the structure of the Hecke rings of $p$-adic Chevalley groups}, Inst. Hautes Études Sci. Publ. Math. No. 25 (1965), 5–48.

\bibitem{Jones87} V. Jones, \emph{Hecke algebra representations of braid groups and link polynomials}, Ann. of Math. (2) 126 (1987), no. 2, 335–388.

\bibitem{KaadKyed21} J. Kaad, D. Kyed, \emph{Dynamics of compact quantum metric spaces}, Ergodic Theory Dynam. Systems 41 (2021), no. 7, 2069–2109.

\bibitem{KazhdanLusztig79} D. Kazhdan, G. Lusztig, \emph{Representations of Coxeter groups and Hecke algebras}, Invent. Math. 53 (1979), no. 2, 165–184.

\bibitem{KazhdanLusztig87} D. Kazhdan, G. Lusztig, \emph{Proof of the Deligne-Langlands conjecture for Hecke algebras}, Invent. Math. 87 (1987), no. 1, 153–215.

\bibitem{Kerr03} D. Kerr, \emph{Matricial quantum Gromov-Hausdorff distance}, J. Funct. Anal. 205 (2003), no. 1, 132–167.

\bibitem{KerrHanfeng09} D. Kerr, H. Li, \emph{On Gromov-Hausdorff convergence for operator metric spaces}, J. Operator Theory 62 (2009), no. 1, 83–109.

\bibitem{Klisse23-1} M. Klisse, \emph{Topological boundaries of connected graphs and Coxeter groups}, J. Operator Theory 89 (2023), no.2, 429--476.

\bibitem{Klisse23-2} M. Klisse, \emph{Simplicity of right-angled Hecke C$^{\ast}$-algebras}, Int. Math. Res. Not. IMRN(2023), no.8, 6598--6623.

\bibitem{Klisse25} M. Klisse, \emph{Universal C$^{\ast}$-algebras from graph products: structure and applications}, arXiv preprint arXiv:2507.12271 (2025).

\bibitem{Kumar02} S. Kumar, \emph{Kac-Moody groups, their flag varieties and representation theory}, Progr. Math., 204
Birkhäuser Boston, Inc., Boston, MA, 2002, xvi+606 pp.

\bibitem{Latremoliere15} F. Latrémolière, \emph{The dual Gromov-Hausdorff propinquity}, J. Math. Pures Appl. (9) 103 (2015), no. 2, 303–351.

\bibitem{Latremoliere16} F. Latrémolière, \emph{The quantum Gromov-Hausdorff propinquity}, Trans. Amer. Math. Soc. 368 (2016), no. 1, 365–411.

\bibitem{Latremoliere19} F. Latrémolière, \emph{The modular Gromov-Hausdorff propinquity}, Dissertationes Math. 544 (2019), 70 pp.

\bibitem{Latremoliere22} F. Latrémolière, \emph{The Gromov-Hausdorff propinquity for metric spectral triples}, Adv. Math. 404 (2022), Paper No. 108393, 56 pp.

\bibitem{Li03} H. Li, \emph{C$^{\ast}$-algebraic quantum Gromov-Hausdorff distance}, arXiv preprint arXivmath/0312003 (2003).

\bibitem{Li06} H. Li, \emph{Order-unit quantum Gromov-Hausdorff distance}, J. Funct. Anal. 231 (2006), no. 2, 312–360.

\bibitem{Lusztig89} G. Lusztig, \emph{Affine Hecke algebras and their graded version}, J. Amer. Math. Soc. 2 (1989), no. 3, 599--635.

\bibitem{Matsumoto77} H. Matsumoto, \emph{Analyse harmonique dans les systèmes de Tits bornologiques de type affine}, Lecture Notes in Mathematics, Vol. 590, Springer-Verlag, Berlin-New York, 1977. i+219 pp.

\bibitem{Moussong88} G. Moussong, \emph{Hyperbolic Coxeter groups}, Thesis (Ph.D.)-The Ohio State University. 1988, 55 pp.

\bibitem{Opdam04} E. M. Opdam, \emph{On the spectral decomposition of affine Hecke algebras}, J. Inst. Math. Jussieu 3 (2004), no. 4, 531–648.

\bibitem{OpdamSolleveld10} E. M. Opdam, M. Solleveld, \emph{Discrete series characters for affine Hecke algebras and their formal degrees}, Acta Math. 205 (2010), no. 1, 105–187.

\bibitem{OzawaRieffel05} N. Ozawa, M. A. Rieffel, \emph{Hyperbolic group C$^{\ast}$-algebras and free-product C$^{\ast}$-algebras as compact quantum metric spaces}, Canad. J. Math. 57 (2005), no. 5, 1056--1079.

\bibitem{Radulescu94} F. Radulescu, \emph{Random matrices, amalgamated free products and subfactors of the von Neumann algebra of a free group, of noninteger index}, Invent. Math. 115 (1994), no. 2, 347–389.


\bibitem{RaumSkalski22} S. Raum, A. Skalski, \emph{Classifying right-angled Hecke C$^\ast$-algebras via K-theoretic invariants}, Adv. Math. 407 (2022), Paper No. 108559, 24 pp.

\bibitem{RaumSkalski23} S. Raum, A. Skalski, \emph{Factorial multiparameter Hecke von Neumann algebras and representations of groups acting on right-angled buildings}, J. Math. Pures Appl. (9) 172 (2023), 265--298.


\bibitem{Remy12} B. Rémy, \emph{Buildings and Kac-Moody groups, buildings, finite geometries and groups}, 231–250, Springer Proc. Math., 10, Springer, New York, 2012.

\bibitem{Rieffel98} M. A. Rieffel, \emph{Metrics on states from actions of compact groups}, Doc. Math. 3 (1998), 215--229.

\bibitem{Rieffel99} M. A. Rieffel, \emph{Metrics on state spaces}, Doc. Math. 4 (1999), 559--600.

\bibitem{Rieffel02} M. A. Rieffel, \emph{Group C$^{\ast}$-algebras as compact quantum metric spaces}, Doc. Math. 7 (2002), 605--651.

\bibitem{Rieffel04} M. Rieffel, \emph{Gromov--Hausdorff distance for quantum metric spaces}, Mem. Amer. Math. Soc. 168 (2004), no. 796, 1–65.

\bibitem{Solleveld12} M. Solleveld, \emph{On the classification of irreducible representations of affine Hecke
algebras with unequal parameters}, Represent. Theory 16 (2012), 1–87.

\bibitem{Solleveld18} M. Solleveld, \emph{Topological K -theory of affine Hecke algebras}, Ann. K -Theory 3 (2018), no. 3, 395–460.

\bibitem{ToyotaYang23} R. Toyota, Zhiyuan Yang, \emph{An operator-valued Haagerup inequality for hyperbolic groups}, arXiv preprint arXiv:2311.13651 (2023).

\bibitem{Voiculescu85} D. Voiculescu, \emph{Symmetries of some reduced free product C$^{\ast}$-algebras}, Operator algebras and their connections with topology and ergodic theory (Bu\c{s}teni, 1983), 556--588. Lecture Notes in Math., 1132 Springer-Verlag, Berlin, 1985. 

\bibitem{Voiculescu90} D. Voiculescu, \emph{On the existence of quasicentral approximate units relative to normed ideals, I.}, J. Funct. Anal. 91 (1990), no. 1, 1–36.

\end{thebibliography}
\end{document}